\documentclass[amstex,12pt,reqno]{amsart}
\usepackage{amsmath, amssymb, amsthm}
\usepackage{mathrsfs}
\usepackage[all]{xy}
\usepackage{tikz}
\usetikzlibrary{intersections,calc,arrows.meta,patterns}

\newtheorem{theorem}{Theorem}[section]
\newtheorem{lemma}[theorem]{Lemma}
\newtheorem{proposition}[theorem]{Proposition}
\newtheorem{corollary}[theorem]{Corollary}
\newtheorem{claim}{Claim}

\theoremstyle{definition}
\newtheorem*{definition}{Definition}
\newtheorem*{remark}{Remark}

\numberwithin{equation}{section}

\title[Chordal Loewner chains and Teichm\"uller spaces on the half-plane]
{Chordal Loewner chains and Teichm\"uller spaces on the half-plane}

\author[H. Wei]{Huaying Wei} 
\address{Department of Mathematics and Statistics, Jiangsu Normal University \endgraf Xuzhou 221116, PR China} 
\curraddr{Department of Mathematics, School of Education, Waseda University \endgraf
Shinjuku, Tokyo 169-8050, Japan}
\email{hywei@jsnu.edu.cn} 

\author[K. Matsuzaki]{Katsuhiko Matsuzaki}
\address{Department of Mathematics, School of Education, Waseda University \endgraf
Shinjuku, Tokyo 169-8050, Japan}
\email{matsuzak@waseda.jp}

\makeatletter
\@namedef{subjclassname@2020}{%
\textup{2020} Mathematics Subject Classification}
\makeatother

\thanks{Research supported by 
Japan Society for the Promotion of Science (KAKENHI 18H01125 and 21F20027).}

\begin{document}

\maketitle

\begin{abstract}
We consider a univalent analytic function $f$ on the half-plane
satisfying the condition that the supremum norm of its (pre-)Schwarzian derivative vanishes on
the boundary. Under certain extra assumptions on $f$, we show that there exists a chordal Loewner chain 
initiated from $f$ until some finite time, and this Loewner chain defines a
quasiconformal extension of $f$ over the boundary such that its complex dilatation is given
explicitly in terms of the (pre-)Schwarzian derivative in some neighborhood of the boundary. 
This can be regarded as the half-plane version of
the corresponding result developed on the disk by Becker and also the generalization of
the Ahlfors--Weill formula. As an application of this quasiconformal extension,
we complete the characterization of an element of the VMO-Teichm\"uller space on the half-plane
using the vanishing Carleson measure condition induced by the (pre-)Schwarzian derivative.
\end{abstract}

\section{Introduction}

Disks and half-planes in the complex plane are conformally equivalent. Most of the results in complex analysis
have no difference between these cases, and thus one may easily overlook the situations
where the difference occurs when boundary conditions are imposed on functions of these domains; one boundary is compact and the other is non-compact.
In active research areas of complex analysis interacting with each other, we can find these phenomena in recent studies.
One is a Loewner chain and the other is a Teichm\"uller space. In this paper, we focus on a Loewner chain based on 
a point on a non-compact
boundary, and consider its application to a Teichm\"uller space whose functions satisfy a certain vanishing condition on 
the non-compact boundary.

The idea of the Loewner theory, first introduced by Loewner \cite{Loe23} in 1923  
and later developed by Kufarev \cite{Kuf43} in 1943 and Pommerenke \cite{Pom65} in 1965,
is to embed a univalent function 
into a family of time-parametrized univalent functions (nowadays known as a Loewner chain) satisfying a suitable differential equation. 
Subsequently, the parametric representation method of univalent functions has been widely applied and developed in the theory of univalent functions as well as in other fields of mathematics. These include its stochastic analogue (Schramm--Loewner Evolution) discovered by Schramm \cite{Sch00} in 2000, which has attracted substantial attention in probability theory and conformal field theory. 
The most remarkable result in classical complex analysis is that the famous Bieberbach conjecture was solved by de Branges \cite{DB85} in 1985 with the use of the extension of Loewner's original approach. 

In this paper, our interest is mainly devoted to an application of the Loewner differential equation of chordal type to quasiconformal extensions of univalent functions on the right half-plane $\mathbb H := \{z = x + iy \in \mathbb C:x > 0 \}$. The chordal case was less well known until a decade ago, while the radial case 
is the main focus of the classical Loewner theory.   

Pommerenke applied successfully the radial Loewner differential equation to show the univalence criteria of analytic functions on the unit disk $\mathbb D := \{z \in \mathbb C: |z| < 1\}$  (see \cite[Theorem 6.2]{Pom75}). Particularly, it is worth mentioning that many sufficient conditions for univalence can be deduced in this way (see \cite[Theorem 5.3]{Be80}). Moreover, most of the univalence criteria can be refined to quasiconformal extension criteria. 

We first recall the following well-known result prior to the application of the Loewner theory.  
Let $f$ be analytic in a disk $D$ in the extended complex plane $\overline{\mathbb C}$, and assume that
\begin{equation}\label{Ssup1}
\sup_{z \in  D} \frac{1}{2}\rho_D^{-2}(z) |Sf(z)| \leq k. 
\end{equation}
Here, $\rho_D$ denotes the hyperbolic density on $D$ (of the Gaussian curvature $\equiv -4$),
which satisfies $\rho_{D}(\varphi(z))|\varphi'(z)| = 1/(2{\rm Re}z)$ for a conformal mapping $\varphi$ of $\mathbb H$ onto $D$;
$Sf = (Pf)' - (Pf)^2/2$ is the Schwarzian derivative of $f$ given by the pre-Schwarzian derivative
$Pf := f''/f'$. If $k = 1$ then $f$ is univalent in $D$, which is a sharp univalence criterion
proved by Nehari \cite{Ne49} in 1949; 
if $k < 1$ then $f$ admits a $k$-quasiconformal extension to $\overline{\mathbb C}$ as Ahlfors and Weill \cite{AW62} found in 1962
(see also \cite[p.87]{Le}). In fact, 
when $D=\mathbb D$, the complex dilatation $\mu(z)$ of this extension has the form 
\begin{equation}\label{Smu1}
\mu(1/\bar z) = -\frac{1}{2} \left( z/\bar z \right)^2 (1 - |z|^2)^2 Sf(z);
\end{equation}
when $D=\mathbb H$, it has the form
\begin{equation}\label{Smu2}
\mu(-\bar z) = - \frac{1}{2}(2{\rm Re}z)^2  Sf(z).
\end{equation}

These results can be also obtained by using the Loewner chain.
In 1972, Becker \cite{Be72} discovered  
the method of showing the quasiconformal extendibility of univalent functions in $\mathbb D$ by means of the radial Loewner differential equation: let 
$(f_t)_{t \geq 0}$ be a univalent solution of 
\begin{equation*}\label{PDE4}
\dot{f}_t (z) =  z p(z, t) f_t'(z)
\end{equation*}
for all $z \in \mathbb D$ and almost every $t \geq 0$ with some Herglotz function $p(z, t)$.   
He showed that if the image domain $p(\mathbb D, t)$ 
lies in a closed disk 
\begin{equation*}
U(k):=\left\{w \in \mathbb H: \; \left| w - \frac{1+k^2}{1-k^2} \right| \leq \frac{2k}{1-k^2} \right\} 
\end{equation*} 
for almost every time $t \geq 0$ and for some 
$k < 1$ independent of $t$,
then each function $f_t$ has a quasiconformal extension to the whole plane $\mathbb C$. 
More recently, Gumenyuk and Hotta \cite{GH} developed this method in the chordal case by
considering the corresponding equation
\begin{equation*}\label{PDE5}
\dot{f}_t (z) = -  p(z, t) f_t'(z),
\end{equation*}
and opened a way of
quasiconformal extensions of analytic functions defined on $\mathbb H$.
The advantage of the Loewner theory is that these extensions are yielded explicitly
by tracing the boundary values of the Loewner chain.

There is another advantage of using the Loewner chain. 
As the existence of a quasiconformal extension depends only on the boundary curve of the image domain, it is natural to consider 
the limit at the boundary
corresponding to condition \eqref{Ssup1}. 
In this case, 
one has to assume first that $f$ is already known to be univalent having a Jordan domain as its image. 
Towards this direction, we refer to Becker's work (see \cite{Be80-1, Be80}). 
As we mentioned above, he obtained a method for global quasiconformal extendibility.
Moreover, this implies that if the time-variable $t$ is taken only over an interval $[0, \tau)$ 
then the extension only to the disk $\{z \in \mathbb C:|z| < R \}$ for some $R > 1$ can be obtained. 
Based on that, the following sufficient condition for quasiconformal extendibility of univalent analytic functions was proved. See \cite[Theorem 1]{Be80-1}, and also \cite[Theorem 5.4]{Be80} and \cite[Theorem 3]{Be87}.

\begin{theorem}[Becker] \label{Beck}
Let $f$ be univalent and analytic in $\mathbb D$ such that $f(\mathbb D)$ is a Jordan domain. If 
\begin{equation}\label{Slimsup1}
\limsup_{|z| \to 1^-} \frac{1}{2}(1 - |z|^2)^2 |Sf(z)| \leq k \;(< 1), 
\end{equation}
then $f$ admits a quasiconformal extension to the disk $\{z\in \mathbb C: |z| < R \}$ for some $R > 1$ whose complex dilatation 
$\mu$ has the form in \eqref{Smu1}.
\end{theorem}

We note that under the same assumption of the above theorem, $f$ given on a smaller disk 
$\{z\in \mathbb C: |z| < R-\varepsilon \}$
for any $\varepsilon>0$
can be further extended quasiconformally to $\mathbb C$
by a general fact of quasiconformal extension as in \cite[Theorem II.1.8]{LV}.

In this paper, we will give, by means of the chordal Loewner differential equation, the half-plane analogue of Theorem \ref{Beck},
which is also the boundary limit version in the half-plane case of the Ahlfors--Weill extension:
\begin{theorem}\label{main1}
Let $h$ be univalent and analytic in $\mathbb H$ with $\lim_{z \to \infty}h(z) = \infty$ such that $h(\mathbb H)$ is a Jordan domain. 
Let $h$ satisfy the condition
\begin{equation}\label{Slimsup2}
\lim_{{\rm Re} z\to 0^+} \frac{1}{2}(2{\rm Re} z)^2 \left|Sh(z)\right| = 0.
\end{equation}
Suppose further that either  
\begin{enumerate}
\item[$\mathbf{(A)}$] both $h$ and $h^{-1}$ are locally uniformly continuous in a neighborhood of $\infty$,  or 
\item[$\mathbf{(B)}$] $h$ is locally quasiconformally extendible to a neighborhood of $\infty$.
\end{enumerate}
Then, the function
\begin{equation}\label{extension}
\hat h(z) = 
\begin{cases}
h(z), & \text {if}\;\; {\rm Re}\, z \geq 0,\\
h(z^*) + \frac{(2{\rm Re} z) h'(z^*)}{1 - ({\rm Re}\, z) Ph(z^*)}  \quad (z^* := -\bar z),    & \text{if}\;\;  -\tau < {\rm Re}\, z < 0,
\end{cases}
\end{equation}
defines a quasiconformal extension of $h$ over i$\mathbb R$ to 
$\mathbb{H}^{*}_{(0,\tau)} := \{z \in \mathbb C: -\tau < {\rm Re}\, z < 0 \}$ for some $\tau > 0$  
such that its complex dilatation $\mu$ on $\mathbb{H}^{*}_{(0,\tau)}$ has the form 
\begin{equation}\label{Smu3}
\mu(z) = - \frac{1}{2}(2{\rm Re}z)^2  Sh(z^*).
\end{equation} 
\end{theorem}

Condition $\mathbf{(A)}$ is a practical assumption. For instance, if $h$ satisfies the hydrodynamic normalization at $\infty$
uniformly in $\mathbb H$
(see \cite{GB}),
then this condition follows. Condition $\mathbf{(B)}$ is an a priori assumption in the sense that if $h$ admits such a
quasiconformal extension as in the conclusion of the theorem, then this condition must be satisfied.

We remark that if assumptions $\mathbf{(A)}$ and $\mathbf{(B)}$ are not imposed then the assertion of Theorem \ref{main1} 
does not hold true. This will be verified by constructing an explicit example in Corollary \ref{example}. 
This, in particular, illustrates that the case of half-plane shown in Theorem \ref{main1} represents a more complicated object
where we have to deal with compactness problems on the boundary that do not arise in the case of disk.
In Section \ref{vanishnorm}, we develop the arguments for this example. 

We demonstrate the proof of Theorem \ref{main1} in Section \ref{mainproof} by relying on the Loewner theory. 
In Section \ref{Loewner}, we collect several necessary definitions and results from this theory which will be used in the proofs
as well as a concise summary of the Loewner theory from a modern viewpoint.

Under the circumstances of Theorem \ref{main1},
a family of analytic functions initiated from $h$ is defined canonically by
\begin{equation*}\label{family0}
h_t(z)=h(z+t)-\frac{2th'(z+t)}{1+tPh(z+t)}
\end{equation*}
for $z \in \mathbb H$ and for $t$ within a limited period of time starting at $0$. 
An important step for the proof of Theorem \ref{main1} is to show that
$h_t$ is univalent when $t \geq 0$ is sufficiently small, which turns out to be a Loewner chain
over an interval $[0,\tau)$ for some $\tau>0$. Gumenyuk and Hotta \cite{GH} carried out this under the circumstances where the time $t$ can be extended to $+\infty$, but their argument does not work for the case of limited time period. We remark that, differently from the case of half-plane, the corresponding argument for Theorem \ref{Beck} is essentially the same as that for the case where the time
$t$ extends to $+\infty$.  
Our method of showing that $h_t$ is univalent in a small interval of time involves novel techniques. This is achieved in Theorem \ref{univalentA} under assumption $\mathbf{(A)}$ and 
Theorem \ref{univalentB} under assumption $\mathbf{(B)}$. 

For the proof, we need to show that
condition \eqref{Slimsup2} is equivalent to
\begin{equation}\label{Plimsup2}
\lim_{{\rm Re} z\to 0^+} (2{\rm Re} z) \left|Ph(z)\right| = 0.
\end{equation}
In Section \ref{vanishnorm}, we prove this equivalence in Theorem \ref{PS} which seems to be natural, but is missing in the literature. 
Moreover, by replacing condition \eqref{Slimsup2} with \eqref{Plimsup2}, we can also obtain the analogous result as Theorem \ref{Pmain12} to
Theorem \ref{main1}, which represents the complex dilatation $\mu$ in terms of $Ph$.
This is stated briefly in Section \ref{preversion}. 

As a special case of Theorem \ref{main1} and also of Theorem \ref{Pmain12},
the following form can be used conveniently for subspaces of the universal Teichm\"uller space because 
univalent analytic functions are already assumed to be quasiconformally extendible to $\overline{\mathbb C}$ in these settings.
The importance of this result lies not on the quasiconformal extendibility of course, but on
the representation of the complex dilatation.

\begin{corollary}\label{main2}
Let $h$ be univalent and analytic in $\mathbb H$ that is quasiconformally extendible to $\mathbb C$
with $\lim_{z \to \infty}h(z) = \infty$.
If $h$ satisfies \eqref{Slimsup2}, then the assertion in Theorem \ref{main1} holds, that is,
it has the extension whose complex dilatation is of the form in \eqref{Smu3} in a strip domain $\mathbb{H}^{*}_{(0,t)}$
for some $t>0$
over $i\mathbb R$,
and further extendible to $\mathbb C$ keeping this initial extension.
\end{corollary}

Thanks to the fact that the complex dilatation $\mu$ of the quasiconformal extension $\hat h$ in Theorem \ref{main1} is expressed explicitly by the Schwarzian derivative $Sh$ of $h$, and the way of constructing $\hat h$ by means of the boundary values of a suitable Loewner chain, 
the above
Corollary \ref{main2} can be neatly used to give an appropriate quasiconformal extension for the subspaces of the universal Teichm\"uller space smaller than the so-called little universal Teichm\"uller space on the half-plane. 
In this paper, we deal with the VMO-Teichm\"uller space as such an example.
 
The universal Teichm\"uller space $T$ can be regarded as the set of all conformal mappings $h$ (up to the composition of a M\"obius transformation) on $\mathbb H$ which can be extended to a quasiconformal homeomorphism of $\overline{\mathbb C}$. 
Subspaces of $T$ are obtained by imposing some conditions on $h$ in general.
If this condition is conformally invariant, then there is no difference between the cases
where we consider it on $\mathbb D$ and on $\mathbb H$. The Weil--Petersson Teich\"uller space and the BMO-Teichm\"uller space are such examples, which have received much attention over the years (see \cite{AZ, Bi, Sh, SW, WM-2, WM-3} and references therein). 

On the contrary, if the condition is given on the boundary, for instance, the vanishing condition of the supremum norm of Schwarzian derivative as for the little universal Teichm\"uller space, Teichm\"uller spaces 
on $\mathbb D$ and on $\mathbb H$
are different. Becker and Pommerenke \cite{BP} worked on this for $\mathbb D$, and later
Hu, Wu and Shen \cite{HWS} considered this on $\mathbb H$.

The VMO-Teichm\"uller space is the subspace of the BMO-Teichm\"uller space defined by certain vanishing condition on the boundary (see \cite{SW}). If this space is 
defined on the half-plane $\mathbb H$, it is the set of all $h \in T$ satisfying $\log h' \in {\rm VMOA}(\mathbb H)$, the space of analytic functions in $\mathbb H$ of vanishing mean oscillation. In Section \ref{app},  
as an application of  Corollary \ref{main2}, we will prove the following result. 
This fills in the missing part left in the study of various models of VMO-Teichm\"uller space
on the half-plane (see \cite[Theorem 2.2]{Sh}, \cite{WM-2}). 

\begin{theorem}\label{main3}
Let $h$ be univalent and analytic in $\mathbb H$ such that
$h$ is quasiconformally extendible to $\overline{\mathbb C}$ with $h(\infty) = \infty$. 
If $\log h' \in {\rm VMOA}(\mathbb H)$, then $h$ admits a quasiconformal extension to $\overline{\mathbb C}$ 
such that its complex dilatation $\tilde \mu$ induces a vanishing Carleson measure 
\begin{equation}\label{carleson}
\lambda_{\tilde \mu} := \frac{|\tilde \mu(z)|^2}{(-2{\rm Re}z)}dxdy 
\end{equation}
of the left half-plane 
$\mathbb H^* := \{z = x + iy \in \mathbb C: x < 0\}$.
\end{theorem}

\section{Preliminaries on the Loewner theory}\label{Loewner}
In this section, we give several basic definitions and results on the (generalized) Loewner theory, 
proposed by Bracci, Contreras and D\'iaz-Madrigal \cite{BCDM09, BCDM12}, which allows us to 
treat evolution families with inner fixed points (the radial case) and with boundary fixed points (the chordal case)
at the same time. 
This unified theory relies partially on the theory of one parametric semigroups, which is actually 
the autonomous version of the Loewner theory. 

\subsection{Generalized Loewner theory}
Let $D \subset \overline{\mathbb C}$ be a simply connected domain conformally equivalent to $\mathbb D$ and $\mathbb H$. 
We denote the family of all analytic functions on $D$ by 
${\rm Hol}(D, \mathbb C)$, and the family of all analytic self-maps of $D$ by ${\rm Hol}(D)$. 

Let $\mathcal U \subset {\rm Hol}(\mathbb D)$ be a semigroup with respect to the operation of composition containing 
the identity map ${\rm id}_{\mathbb D}$. A family 
$(\phi_t)_{t \geq 0}$ in $\mathcal U$ such that $\phi_0 = {\rm id}_{\mathbb D}$, $\phi_t\circ\phi_s = \phi_{t+s}$ for any $s, t \geq 0$, and $\phi_t(z) \to z$ locally uniformly on $ \mathbb D$ as $t \to 0$, is called a continuous one-parameter semigroup. 
It is known \cite{BP78} that for any such semigroup $(\phi_t)$ there exists a function $G\in {\rm Hol}(\mathbb D, \mathbb C)$ such that for each $z \in \mathbb D$ the function $\phi_t(z)$ is the unique solution of the initial value problem 
\begin{equation}\label{ODE1}
\frac{dw(z,t)}{dt} = G(w(z,t)), \quad w(z,0) = z.
\end{equation}
The function $G$ is called the infinitesimal generator of $(\phi_t)$.
A criterion for a function $G \in {\rm Hol}(\mathbb D, \mathbb C)$ to be an infinitesimal generator of some continuous one-parameter
semigroup
is the following Berkson--Porta representation: 
\begin{equation}\label{BP1}
G(z) = (\tau - z)(1 - \bar\tau z)\,p(z)
\end{equation}
for a point $\tau \in \overline{\mathbb D}$ and a function $p \in {\rm Hol}(\mathbb D, \mathbb C)$ with ${\rm Re}\,p(z) \geq 0$ for all $z \in \mathbb D$. Moreover, if $G \not\equiv 0$, then such a representation is unique. In fact, the point $\tau$ is the common Denjoy--Wolff point (see below for its definition) of all $\phi_t$'s that are different from ${\rm id}_{\mathbb D}$.  
Such a correspondence between the continuous one-parameter semigroup $(\phi_t)$ and the infinitesimal generator $G$ is one-to-one.

Now we introduce the Herglotz vector field in $\mathbb D$, which can be regarded as a time-dependent infinitesimal generator. One of its equivalent definitions is as follows (see \cite[Theorem 4.8]{BCDM12}).  
A Herglotz vector field in $\mathbb D$ is a map $G:\; \mathbb D \times [0, +\infty) \to \mathbb C$ of the form 
\begin{equation}\label{BP2}
G(z, t) = (\tau(t) - z)(1 - \overline{\tau(t)}z)\, p(z, t)
\end{equation}
for all $z \in \mathbb D$ and almost every $t \in [0, +\infty)$, where $\tau:[0, +\infty) \to \overline{\mathbb D}$ is a measurable function and $p:\mathbb D \times [0, +\infty) \to \mathbb C$ is a Herglotz function defined as follows.
\begin{definition}
A Herglotz function on $\mathbb D$ is a map $p:\mathbb D \times [0, +\infty) \to \mathbb C$ satisfying the following conditions: 
\begin{enumerate}
\item[(HF1)] $p(z, \cdot)$ is locally integrable on $[0, +\infty)$ for all $z \in \mathbb D$;
\item[(HF2)] $p(\cdot, t)$ is analytic on $\mathbb D$ for almost every $t \in [0, +\infty)$ and ${\rm Re}\,p(\cdot, t) \geq 0$. 
\end{enumerate}
\end{definition}
Moreover, \cite[Theorem 4.8]{BCDM12} asserts that if two couples $(p_1,\tau_1)$ and $(p_2,\tau_2)$ 
generate the same $G(z,t)$ up to a set of measure zero on the $t$-axis, then 
$p_1(z,t) = p_2(z,t)$ for all $z \in \mathbb D$ and almost every $t \in [0, +\infty)$, 
and $\tau_1(t) = \tau_2(t)$ for almost every $t \in [0, +\infty)$ satisfying $G(\cdot,\, t) \not\equiv 0$. 

Furthermore, it is known that the evolution family $(\varphi_{s,t})_{t\geq s\geq 0}$ is the unique solution of the following initial value problem for the generalized Loewner--Kufarev ODE that is driven by the Herglotz vector field $G$ 
(see \cite[Section 2.1]{CDMG13}, \cite[Chapter 18]{Kur}): 
\begin{equation}\label{ODE2}
\frac{d\varphi_{s,t}(z)}{dt} = G(\varphi_{s,t}(z),\, t),  \quad \varphi_{s,s}(z) = z
\end{equation}
for any initial point $z \in \mathbb D$, any starting time $s \geq 0$ and almost every $t \geq s$.  
This equation establishes a one-to-one correspondence between the evolution families $(\varphi_{s,t})$ and the Herglotz vector fields $G$ up to a set of measure zero on the $t$-axis (see \cite[Theorem 1.1]{BCDM12}).  
Comparing \eqref{ODE1} and \eqref{ODE2}, we can regard an evolution family as a non-autonomous analogue of a continuous one-parameter semigroup. 

An independent definition of an evolution family can be given as follows:
\begin{definition}[{\cite{BCDM12}}]
An evolution family in $\mathbb D$ is a two-parameter family $(\varphi_{s,t})_{t\geq s\geq 0} \subset {\rm Hol}(\mathbb D)$ satisfying the following conditions:
\begin{enumerate}
\item[(EF1)] $\varphi_{s,s} = {\rm id}_{\mathbb D}$ for all $s \geq 0$;
\item[(EF2)] $\varphi_{s,t} = \varphi_{u,t} \circ \varphi_{s,u}$ whenever $0 \leq s \leq u \leq t < +\infty$;
\item[(EF3)] for each $z \in \mathbb D$ there exists a non-negative locally integrable function $k_z$ on $[0, \, +\infty)$ such that 
$|\varphi_{s,u}(z) - \varphi_{s,t}(z)| \leq \int_u^t k_z (\xi) d\xi$ whenever $0 \leq s \leq u \leq t < +\infty$.
\end{enumerate}
\end{definition}

For all $0 \leq s  \leq t < +\infty$, $\varphi_{s,t}$ is univalent in $\mathbb D$, which follows from the uniqueness of the solution of \eqref{ODE2} (see \cite[Corollary 6.3]{BCDM12}) even though the definition does not require elements of an evolution family to be univalent.  The notion of a Loewner chain in the same framework as an evolution family can be given as follows. 
\begin{definition}[{\cite{CDMG10}}]
A family $(f_t)_{t \geq 0} \subset {\rm Hol}(\mathbb D,\,\mathbb C)$ is called a Loewner chain if it satisfies  the following conditions:
\begin{enumerate}
\item[(LC1)] each function $f_t$ is univalent for all $t \geq 0$;
\item[(LC2)] $f_s(\mathbb D) \subset f_t(\mathbb D)$ whenever $0 \leq s \leq t < +\infty$;
\item[(LC3)] for any compact subset $K \subset \mathbb D$ and any $T > 0$ there exists a non-negative integrable function $k_{K,T}$ on $[0, \, T]$ such that 
$|f_s(z) - f_t(z)| \leq \int_s^t k_{K,T} (\xi) d\xi$ whenever $z \in K$ and $0 \leq s  \leq t \leq T$.
\end{enumerate}
\end{definition}

\begin{remark}
In \cite{BCDM12} and \cite{CDMG10}, the definitions of evolution families and Loewner chains contain an integrability order $d \in [1, +\infty]$. We only need to consider the most general case of order $d = 1$.
\end{remark}

For a given Loewner chain $(f_t)$, the equation $\varphi_{s,t} := f_t^{-1}\circ f_s$ defines an evolution family. Differentiating both sides with respect to $t$ yields that $f_t' (\varphi_{s,t}) \dot{\varphi}_{s,t} + \dot{f}_t(\varphi_{s,t}) = 0$ where $f_t' (z) := d f_t(z)/dz$, $\dot f_t := d f_t/dt$, and $\dot{\varphi}_{s,t} := d \varphi_{s,t}/dt$. Combined with \eqref{ODE2}, this yields the generalized Loewner--Kufarev PDE: 
\begin{equation}\label{PDE2}
\dot{f}_t (z) = - f_t'(z) G(z, t)
\end{equation}
for all $z \in \mathbb D$ and almost every $t \geq 0$. 
Conversely, given an evolution family $(\varphi_{s,t})$, or equivalently, a Herglotz vector field $G$, one can obtain the corresponding Loewner chain $(f_t)$ by \eqref{PDE2}, which turns out to be unique up to the post-composition of a conformal mapping of 
the domain $\bigcup_{t\geq 0} f_t(\mathbb D)$. Moreover, the so-called standard Loewner chain defined in \cite{CDMG10} is determined uniquely, which is a Loewner chain $(f_t)$ such that $f_0(0) = f_0'(0) - 1 = 0$ 
(notice that only $f_0$ is normalized) and $\bigcup_{t\geq 0} f_t(\mathbb D)$
is either $\mathbb C$ or a disk centered at the origin. 

Overall, we have seen the one-to-one correspondences among evolution families $(\varphi_{s,t})$,  Herglotz vector fields $G$,  
couples $(p, \tau)$ of Herglotz functions $p$ and measurable functions $\tau$,
and  Loewner chains $(f_t)$ up to conformal mappings. In particular, the relationship between $(\varphi_{s,t})$ and $(p,\tau)$ is expressed by 
\begin{equation}\label{ODE3}
\frac{d\varphi_{s,t}(z)}{dt} = (\tau(t) - \varphi_{s,t}(z))(1 - \overline{\tau(t)}\varphi_{s,t}(z)) p(\varphi_{s,t}(z), t),  \quad \varphi_{s,s}(z) = z
\end{equation}
for all $z \in \mathbb D$ and almost every $t \geq s \geq 0$. Moreover, the relationship between $(f_t)$ and $(p,\tau)$ is expressed by 
\begin{equation}\label{PDE3}
\dot{f}_t (z) = - f_t'(z) (\tau(t) - z)(1 - \overline{\tau(t)}z) p(z, t)
\end{equation}
for all $z \in \mathbb D$ and almost every $t \geq 0$. We point out that if $\tau(t) \equiv {\rm const}$, then such a constant is 
the Denjoy--Wolff point of $\varphi_{s,t}$ for all 
$0 \leq s \leq t < +\infty$ different from ${\rm id}_{\mathbb D}$ (see \cite[Corollary 7.2]{BCDM12}).
This is the same as in the Berkson--Porta representation for continuous one-parameter semigroups.

We end this subsection with a brief introduction of the Denjoy--Wolff point mentioned above.
The Julia--Wolff--Carath\'eodory theorem (see \cite[Theorem 1.10]{ES}) asserts that for any $\varphi\in {\rm Hol}(\mathbb D)$ different from ${\rm id}_{\mathbb D}$ there exists a unique fixed point $\tau \in \overline{\mathbb D}$ such that $\varphi(\tau) = \tau$ and $|\varphi'(\tau)| \leq 1$. Such $\tau$ is called the Denjoy--Wolff point of $\varphi$. In the case of $\tau \in \partial \mathbb D$, 
these should be understood as an angular limit and an angular derivative denoted by using $\angle\lim$.
Namely, $\varphi(\tau) = \tau$ means $\angle \lim_{z \to \tau} \varphi(z) = \tau$, and this condition in fact
implies that 
$\varphi'(\tau) := \angle\lim_{z \to \tau} (\varphi(z) - \tau)/(z - \tau)$ exists in $(0,\, +\infty]$ (see \cite[Proposition 4.13]{Pom}). 
Thus, if $\tau \in \partial\mathbb D$ is the Denjoy--Wolff point of $\varphi$ then $0 < \varphi'(\tau) \leq 1$. 
In the case of $\tau \in \mathbb D$, the assertion $|\varphi'(\tau)| \leq 1$ follows from $\varphi(\tau) = \tau$ by the Schwarz--Pick lemma.  

The Denjoy--Wolff theorem asserts that, excluding a trivial exception where $\varphi$ is ${\rm id}_{\mathbb D}$ or an elliptic automorphism of $\mathbb D$, the sequence $\{\varphi_n\}$ of iterates of $\varphi$, defined as $\varphi_n := \varphi \circ \varphi_{n-1}$, $\varphi_1 = \varphi$, converges to the Denjoy--Wolff point $\tau$ uniformly on compact subsets of $\mathbb D$. This is a prototype of hyperbolic dynamics and contributed much to modern studies on dynamical systems (see \cite[Appendix G]{ALR}).

\subsection{(Classical) radial Loewner theory}
Hereafter, we consider the case that $\tau(t)$ is a constant function, and thus it is the common interior Denjoy--Wolff point $\tau$
of the evolution family $(\varphi_{s,t})$ in $\mathbb D$. When $\tau \in \mathbb D$,
we can assume that 
$\tau$ is $0$ by  the conjugation 
of an automorphism of $\mathbb D$ sending $\tau$ to $0$. 

\begin{definition}
An evolution family $(\varphi_{s,t})$ in $\mathbb D$ is said to be of radial type, if for all $0 \leq s \leq t < +\infty$,  
all elements $\varphi_{s,t}$'s different from ${\rm id}_{\mathbb D}$ share the common interior Denjoy--Wolff point at $0$.  
A Loewner chain $(f_t)$ in $\mathbb D$ is said to be of radial type if the corresponding evolution family $(\varphi_{s,t}) := (f_t^{-1}\circ f_s)$ is of radial type. 
\end{definition}

In this case, we see by the equation $\varphi_{s,t} = f_t^{-1}\circ f_s$ 
that a radial Loewner chain $(f_t)$ is actually a Loewner chain such that $f_t(0) = f_0(0)$ for all $t \geq 0$, and moreover,   
the Herglotz vector field is $G(z, t) = - z p(z, t)$ for all $z \in \mathbb D$ and almost all $t \geq 0$.  
In view of this, any radial evolution family $(\varphi_{s,t})$ and the corresponding Loewner chain $(f_t)$ satisfy the following radial Loewner--Kufarev ODE and PDE, respectively: 
\begin{equation}\label{ODE4}
\frac{d\varphi_{s,t}(z)}{dt} =  - \varphi_{s,t}(z) p(\varphi_{s,t}(z), t),  \quad \varphi_{s,s}(z) = z
\end{equation}
for all $z \in \mathbb D$ and almost every $t \geq s \geq 0$,  and  
\begin{equation}\label{PDE4}
\dot{f}_t (z) =  z p(z, t) f_t'(z)
\end{equation}
for all $z \in \mathbb D$ and almost every $t \geq 0$ with some Herglotz function $p$ determined uniquely up to a set of measure zero on the $t$-axis. 

In modern literature, the so-called classical radial evolution families and Loewner chains mean a special kind of evolution families and Loewner chains considered by Pommerenke (see \cite[Chapter 6]{Pom}). 

\begin{definition}
A radial evolution family and a radial Loewner chain are said to be classical if they satisfy extra normalization hypotheses 
$\varphi_{s,t}'(0) = e^{s-t}$ and $f_t(0) = 0$, $f_t'(0) = e^t$ for all $0 \leq s \leq t < +\infty$, respectively.
\end{definition}

Accordingly, the Herglotz function $p$ satisfies the extra condition $p(0, t) = 1$ for almost every $t \geq 0$. Moreover, the condition $f_t'(0) = e^t$ for all $t \geq 0$ ensures $\bigcup_{t\geq 0} f_t(\mathbb D) = \mathbb C$, which gives the one-to-one correspondence between the classical radial evolution family $(\varphi_{s,t})$ and the classical radial Loewner chain $(f_t)$ that is standard. 
Moreover, $(f_t)$ can be recovered from $(\varphi_{s,t})$ by the equation $f_s = \lim_{t \to +\infty}e^t\varphi_{s,t}$.

\subsection{Chordal Loewner theory}

Consider an evolution family $(\varphi_{s,t})$  in $\mathbb D$ such that for all $0 \leq s \leq t $,  
all elements $\varphi_{s,t}$'s different from ${\rm id}_{\mathbb D}$ 
share the common boundary Denjoy--Wolff point at $\tau$. 
As before, we can assume that 
$\tau$ is $1$ by the conjugation 
of a rotation of $\mathbb D$. 
Obviously, the associated Herglotz vector field $G(z, t)$ has the form $(1 - z)^2 p(z, t)$ for all $z \in \mathbb D$ and almost all  $t \geq 0$ with some Herglotz function $p$.  

In the literature, the case of boundary fixed points is usually treated in the half-plane model
instead of the unit disk model because there the associated vector field assumes a simpler form (see \eqref{ODE5} below).
Passing from the unit disk $\mathbb D$ to the right half-plane $\mathbb H$ by the Cayley transform $\zeta = H(z) = (1 + z)/(1 - z)$ that sends $1$ to $\infty$, we define a family $\tilde\varphi_{s,t} := H \circ \varphi_{s,t} \circ H^{-1} \in {\rm Hol}(\mathbb H)$ and 
assume the point $\infty$  to be the Denjoy--Wolff point of $\tilde\varphi_{s,t}$.  Then, the solution $\varphi_{s, t}(z)$ of 
the differential equation $dw/dt = (1 - w)^2 p(w, t)$ transforms into the solution $\tilde\varphi_{s, t}(\zeta)$ of 
$dw/dt = \tilde{p}(w, t)$ where $\tilde{p} (\zeta,t) = 2p(H^{-1}(\zeta),\, t)$. 

We say that $(\tilde\varphi_{s,t})$ is an evolution family in $\mathbb H$ if $(\varphi_{s,t})$ is an evolution family in $\mathbb D$. Similarly, we say that
$(\tilde f_{t}) := (H\circ f_t \circ H^{-1})$ is a Loewner chain in $\mathbb H$ if $(f_{t})$ is a Loewner chain in $\mathbb D$; and $\tilde p(\cdot, t) := p(H^{-1}(\cdot), t)$ is a Herglotz function in $\mathbb H$ if $p(\cdot, t)$ is a Herglotz function in $\mathbb D$.

\begin{definition}
An evolution family $(\varphi_{s,t})$ in $\mathbb H$ is said to be of chordal type, if for all $0 \leq s \leq t < \infty$,  all elements $\varphi_{s,t}$'s different from ${\rm id}_{\mathbb H}$ share the common boundary Denjoy--Wolff point at $\infty$.  A Loewner chain $(f_t)$ in $\mathbb H$ is said to be of chordal type if the corresponding evolution family $(\varphi_{s,t}) := (f_t^{-1}\circ f_s)$ is of chordal type. 
\end{definition}

We remark that under these circumstances the half-plane version of the Julia--Wolff--Carath\'eodory theorem
implies that each $\varphi_{s,t}$ satisfies 
$$
{\rm Re}\, \varphi_{s,t} (z) \geq \varphi_{s,t}'(\infty) {\rm Re}\, z
$$
for all $z \in \mathbb H$, where $\varphi_{s,t}'(\infty)$ is the Carath\'eodory angular derivative of $(\varphi_{s,t})$ defined by
$$
\varphi_{s,t}'(\infty) := \angle\lim_{z \to \infty} \frac{\varphi_{s,t} (z)}{z} = \frac{1}{(H^{-1}\circ\varphi_{s,t}\circ H)'(1)} \geq 1.
$$

Based on the above arguments, any chordal evolution family $(\varphi_{s,t})$ and the corresponding Loewner chain $(f_t)$ satisfy the following chordal Loewner--Kufarev ODE and PDE, respectively: 
\begin{equation}\label{ODE5}
\frac{d\varphi_{s,t}(z)}{dt} =  p(\varphi_{s,t}(z), t),  \quad \varphi_{s,s}(z) = z
\end{equation}
for all $z \in \mathbb H$ and almost every $t \geq s \geq 0$,  and  
\begin{equation}\label{PDE5}
\dot{f}_t (z) = -  p(z, t) f_t'(z)
\end{equation}
for all $z \in \mathbb H$ and almost every $t \geq 0$ with some Herglotz function $p$ determined uniquely up to a set of measure zero on the $t$-axis. We say that $(\varphi_{s,t})$ and $(f_t)$ are the evolution family and the Loewner chain associated with the Herglotz function $p$, respectively.

\section{Boundary condition of the Schwarzian derivative}\label{vanishnorm}

In this section, we consider the situation where the norm of the Schwarzian derivative of
a univalent analytic function $g$ on the half-plane $\mathbb H$ vanishes at the boundary.
Differently from the case of $\mathbb D$, this does not necessarily imply that $f$ is quasiconformally extendible
even if $f(\mathbb H)$ is a Jordan domain. We show this fact. We also prove that
the vanishing condition of the norm of the Schwarzian derivative is equivalent to that of the pre-Schwarzian derivative.

We start with the following elementary but crucial observation for the arguments of this section. 

\begin{lemma}\label{ratio}
Let $D := D(1, 1)$ be an open disk in $\mathbb H$ with center $1$ and radius $1$. For any $\epsilon > 0$, there exists a horodisk 
$D' \;(\subset D)$ tangent to $\partial D$ at $0$ such that $\rho_D^{-1}(z)/({\rm Re}\, z) \leq \epsilon$ for all $z \in D\backslash D'$. 
\end{lemma}

\begin{proof}
The hyperbolic density and the distance to the boundary
satisfy the estimates $(4\rho_D(z))^{-1} \leq d(z, \partial D) \leq \rho_D^{-1}(z)$ (see \cite[p.92]{Pom}). 
Then, we consider the ratio $d(z, \partial D)/({\rm Re}z)$ instead of $\rho_D^{-1}(z)/({\rm Re}\, z)$.

For a given constant $a > 1$, we define a curve
$$
\Gamma_a := \left\{z \in D: \; \frac{d(z, \partial D)}{{\rm Re} z}  = \frac{1}{a}\right\}.
$$
By computation, $\Gamma_a$ is an ellipse whose points $z=x+iy$ hold an equation
$$
\frac{\left(x - a/(a+1)\right)^2}{\left( a/(a+1) \right)^2} + \frac{y^2}{(a-1)/(a+1)} = 1.
$$
The ellipse $\Gamma_a$ becomes larger as $a$ increases and tends to $\partial D$ as $a \to \infty$.
Moreover, $\Gamma_a$ is contained in $\overline{D_a}$, where
\begin{equation*}
D_{a} := \left\{z \in D: \; \left|z - a/(a+1)\right| < a/(a+1) \right\}
\end{equation*}
is a horodisk tangent to $\partial D$ at $0$.
Then, we see that 
$d(z, \partial D)/({\rm Re}z) \leq 1/a$ for all $z \in D\backslash D_{a}$.
Consequently, by taking $a = 4/\epsilon$ and setting $D' = D_{4/\epsilon}$, 
we have that $\rho_D^{-1}(z)/({\rm Re} z) \leq \epsilon$ for all $z \in D\backslash D'$. 
\end{proof}

This lemma makes it possible to construct conformal mappings on $\mathbb H$ easily that
satisfy the vanishing condition of their Schwarzian derivatives at the boundary. Here is a general manner for this.

\begin{proposition}\label{weak}
Let $g(\zeta)$ be a conformal mapping on $\mathbb H$.  
Let $D = D(1, 1)$ and let $\zeta = \varphi (z) = 2/(z+1)$ be a M\"obius transformation that maps $\mathbb H$ onto $D$ with $\varphi (\infty) = 0$. Then, the conformal mapping $f := g\circ \varphi$ on $\mathbb H$ satisfies 
\begin{equation}\label{Slimsup}
\lim_{{\rm Re}z \to 0^+} (2{\rm Re}z)^2 |Sf(z)| = 0.
\end{equation}
\end{proposition}

\begin{proof}
Since $g$ is univalent on $\mathbb H$, we have $(2{\rm Re}\,\zeta)^2 |Sg(\zeta)| \leq 6$ for all $\zeta \in \mathbb H$ (see \cite[p. 60]{Le}), and then we have that for all $\zeta \in D$,
\begin{equation*}
\begin{split}
\rho_D^{-2}(\zeta) |Sg(\zeta)| &= (2{\rm Re}\,\zeta)^2 |Sg(\zeta)| \frac{\rho_D^{-2}(\zeta)}{(2{\rm Re}\,\zeta)^2} 
 \leq \frac{3}{2} \left(\frac{\rho_D^{-1}(\zeta)}{{\rm Re}\,\zeta}\right)^2.
\end{split}
\end{equation*}
According to Lemma \ref{ratio}, for any $\epsilon > 0$, there exists a horodisk $D' \; (\subset D)$ tangent to 
$\partial D$ at $0$ such that $\rho_D^{-2}(\zeta) |Sg(\zeta)| \leq \epsilon$ for all $\zeta \in D\backslash D'$. 

Set $H_\epsilon = \varphi^{-1}(D')$, which is a half-plane in $\mathbb H$. As
$$
(2{\rm Re}z)^2 |Sf(z)| = \rho_D^{-2}(\zeta) |Sg(\zeta)|
$$
for all $z \in \mathbb H$ by the invariance under the M\"obius transformation $\varphi$, 
we conclude that for all $z \in \mathbb H\backslash H_\epsilon$, 
$(2{\rm Re}z)^2 |Sf(z)| \leq \epsilon$. 
\end{proof}

As a particular application of this construction, we obtain a counter-example to the statement of Theorem \ref{main1}
when we drop the extra assumptions $\bf (A)$ and $\bf (B)$.

\begin{corollary}\label{example}
There exists a conformal mapping $f$ on $\mathbb H$ satisfying both
\begin{enumerate}
\item[(i)] the image $f(\mathbb H)$ is a Jordan domain, but not a quasidisk; and
\item[(ii)] $\lim_{{\rm Re} z \to 0^+}(2{\rm Re}z)^2 |Sf(z)|  = 0$, that is, \eqref{Slimsup}.
\end{enumerate}
\end{corollary}

\begin{proof}
Consider a half horizontal parallel strip $V$ in $\mathbb H$ with vertices $\pm i\pi/2$, that is, 
$$
V := \{w = u + iv \in \mathbb H: -\pi/2 < v < \pi/2 \}. 
$$
By the Schwarz--Christoffel formula,  
$$
w = g(\zeta) := -i \arcsin \left(\frac{i}{\zeta}\right) = -\log\left( -\frac{1}{\zeta} +\sqrt{1 + \frac{1}{\zeta^2}} \right)
$$
maps $\mathbb H$ conformally onto $V$ with 
$g(0) = \infty$, $g(i) = -i\pi/2$ and $g(-i) = i\pi/2$, where we choose the branch of $\arcsin (i/\zeta)$ so that 
it takes $\pi/2$ at $\zeta = i$. 
Following this, we see the image of $D = D(1, 1)$ under $g$ is a Jordan domain, but not a quasidisk since 
its boundary has a cusp at $\infty$. 

Let $\varphi$ be the M\"obius transformation as in Proposition \ref{weak}. 
We define 
$$
f(z) := g\circ \varphi(z) = -i\arcsin \left( i(z+1)/2 \right).
$$
Then, $f(\mathbb H)=g(D)$ satisfies condition (i). Moreover, we see that the conformal mapping $f$ satisfies condition (ii)
by Proposition \ref{weak}. 
\end{proof}

In the latter half of this section, we compare the boundary behavior of Schwarzian derivatives and pre-Schwarzian derivatives for 
univalent analytic functions on $\mathbb H$. This will be used in the proof of Theorem \ref{main1} later.

First, we state a claim below obtained from the Cauchy integral formula. As in \cite[Lemma 4]{GO},
if $h$ is an analytic function on a domain 
$D \subset \mathbb C$, then
\begin{equation}\label{Lemma4}
\sup_{z \in D} |h'(z)| d(z, \partial D)^2 \leq 4 \sup_{z \in D} |h(z)| d(z, \partial D).
\end{equation}
This yields the implication 
\begin{equation}\label{Lemma40}
\lim_{z \to \partial D} |h(z)| d(z, \partial D) = 0 \quad \Longrightarrow \quad \lim_{z \to \partial D} |h'(z)| d(z, \partial D)^2 = 0.
\end{equation} 
Indeed, by setting $r := d(z, \partial D)/2$ for a fixed $z \in D$, \eqref{Lemma40} easily follows from:
\begin{equation*}
\begin{split}
|h'(z)| &= \left| \frac{1}{2\pi i} \int_{|\zeta - z| = r} \frac{h(\zeta)}{(\zeta - z)^2} d\zeta \right| 
 \leq \frac{1}{r}\sup_{|\zeta - z| = r} |h(\zeta)|\\
&\leq  \frac{2}{d(z, \partial D)} \sup_{|\zeta - z| = r} \left( |h(\zeta)| d(\zeta, \partial D) \right) \sup_{|\zeta - z| = r}\frac{1}{d(\zeta, \partial D)}\\
& = \frac{4}{d(z, \partial D)^2}  \sup_{|\zeta - z| = r} \left( |h(\zeta)| d(\zeta, \partial D) \right). 
\end{split}
\end{equation*}

Suppose that $D$ is a simply connected domain with the hyperbolic density $\rho_D$.  
Since $\rho_D(z) \leq 1/d(z, \partial D) \leq 4\rho_D(z)$, it follows from \eqref{Lemma4} and \eqref{Lemma40} 
that 
$$
\sup_{z \in D} |h'(z)|\rho_D^{-2}(z) \leq 64 \sup_{z \in D} |h(z)|\rho_D^{-1}(z);
$$
$$
\lim_{z \to \partial D} |h(z)|\rho_D^{-1}(z) = 0 \quad \Longrightarrow \quad \lim_{z \to \partial D} |h'(z)|\rho_D^{-2}(z) = 0.
$$
For an analytic and locally univalent function $g$ on $D$,
we can apply these facts to the pre-Schwarzian derivative $Pg$ for $g$ and the Schwarzian derivative $Sg=(Pg)'-(Pg)^2/2$
to obtain the following implications:
\begin{equation}\label{sup}
\begin{split}
\sup_{z \in D} |Pg(z)|\rho_D^{-1}(z) =: \lambda \quad & \Longrightarrow \quad \sup_{z \in D} |Sg(z)|\rho_D^{-2}(z) \leq 64\lambda + \lambda^2/2;\\
\lim_{z \to \partial D} |Pg(z)|\rho_D^{-1}(z) = 0 \quad &\Longrightarrow \quad \lim_{z \to \partial D} |Sg(z)|\rho_D^{-2}(z) = 0. 
\end{split}
\end{equation}

Next, we consider the converse implication.
The following result is essentially contained in Becker \cite{Be87}.

\begin{proposition}\label{PSdisk}
Let $f$ be a univalent analytic function on $\mathbb D$ and let $f(\mathbb D)$ be a Jordan domain. 
Suppose further that $f$ satisfies condition \eqref{Slimsup1}.
For $t \in (0,1)$, set 
\begin{equation*}
\begin{split}
\hat\beta(t) &:= \sup_{1-t \leq |z| < 1} (1 - |z|^2) |Pf(z)|,  \\
\hat\sigma(t) &:= \sup_{1-t \leq |z| < 1} (1 - |z|^2)^2 |Sf(z)|.
\end{split}
\end{equation*}
Then, there is some $t_0 \in (0,1)$ such that
$$
\hat\beta (t^{1+\epsilon}) \leq 4\hat\sigma (t) + 8 t^{\epsilon}
$$
for $0 < t \leq t_0$. Here, $\epsilon>0$ can take an arbitrary positive constant.
\end{proposition}

\begin{proof}
Suppose \eqref{Slimsup1} holds. Then by Theorem \ref{Beck}, $f$ has a quasiconformal extension to 
$\mathbb C$ with complex dilatation $\mu$ satisfying 
\begin{equation}\label{half1}
\hat k(t) := \{\Vert \mu(z) \Vert_{\infty}:\, 1 < |z| \leq 1 + t \} \leq \hat\sigma(t)/2
\end{equation}
for any $0 < t \leq t_0$, where $t_0 < 1$ is some constant. By applying Lehto's majorization method, \cite[Theorem 2]{Be87} implies that
\begin{equation}\label{half2}
\hat\beta(t^{1+\epsilon}) \leq 8 (\hat k(t) + t^{\epsilon}), \quad \epsilon > 0
\end{equation}
for any $0 < t < 1$. Combining \eqref{half1} and \eqref{half2}, we obtain the assertion.
\end{proof}

By the combination of this proposition with Lemma \ref{ratio},
we can obtain the equivalence between the vanishing conditions of the norms for
the Schwarzian derivative and the pre-Schwarzian derivative.

\begin{theorem}\label{PS}
Let $f$ be a univalent analytic function on $\mathbb H$. For $t>0$, set 
\begin{equation*}
\begin{split}
\beta(t) &:= \sup_{0 < {\rm Re}z \leq t} (2{\rm Re}z) |Pf(z)|,  \\
\sigma(t) &:= \sup_{0 < {\rm Re}z \leq t} (2{\rm Re}z)^2 |Sf(z)|.
\end{split}
\end{equation*}
Then, $\beta(t) \to 0$ if and only if $\sigma(t) \to 0$ as $t \to 0$.
\end{theorem}

\begin{proof}
If $\beta(t) \to 0$ as $t \to 0$, then  $\sigma(t) \to 0$ as $t \to 0$ by \eqref{sup}. We will show the opposite implication by means of of Proposition \ref{PSdisk}. 

Now suppose $\sigma(t) \to 0$ as $t \to 0$. We choose $t_0 \in(0, 1)$ such that $\sigma(t_0) \leq 2k \,(< 2)$. 
For an arbitrary boundary point $iy_0 \in i\mathbb R$, 
take the open disk $D := D(1+iy_0, 1) \subset \mathbb H$ of radius $1$ that is tangent to $i\mathbb R$ at $iy_0$.
For any $z$ in the domain $\{z \in D: 0 < {\rm Re}\,z \leq t_0 \}$, obviously we have
\begin{equation}\label{domain1}
\rho_{D}^{-2}(z) |Sf(z)| \leq (2{\rm Re z})^2 |Sf(z)| \leq \sigma(t_0).
\end{equation}
Moreover, by Lemma \ref{ratio}, there exists a horodisk $D' \, (\subset D)$ tangent to $\partial D$ at $iy_0$ such that 
$\rho_D^{-1}(z)/({\rm Re}\, z) \leq (2\sigma(t_0)/3)^{1/2}$ for all $z \in D\backslash D'$. This yields that 
\begin{equation}\label{domain2}
\begin{split}
\rho_{D}^{-2}(z) |Sf(z)| &= (2{\rm Re}z)^2 |Sf(z)| \left(\frac{\rho_D^{-1}(z)}{2{\rm Re}z}\right)^2
\leq 6 \left(\frac{\rho_D^{-1}(z)}{2{\rm Re}z}\right)^2 \leq \sigma(t_0)
\end{split}
\end{equation}
for all $z \in D\backslash D'$. 

Combining \eqref{domain1} and \eqref{domain2}, we can define a function $t_1 := \lambda(t_0) \, (\leq t_0)$ tending to $0$ monotonically and continuously as $t_0 \to 0$ such that 
\begin{equation}\label{relation1}
\hat\sigma_D(t_1) := \sup_{0 < d(z, \partial D) \leq t_1 \atop z \in D} \rho_{D}^{-2}(z) |Sf(z)| \leq \sigma(t_0).
\end{equation}
Similarly, we define 
$$
\hat\beta_D(t_1) := \sup_{0 < d(z, \partial D) \leq t_1 \atop z \in D} \rho_{D}^{-1}(z) |Pf(z)|.
$$
By Proposition \ref{PSdisk} for $\epsilon = 1$, we have 
\begin{equation}\label{relation2}
\hat\beta_D (t_1^2) \leq 4\hat\sigma_D (t_1) + 8 t_1 
\end{equation}
for all sufficiently small $t_1>0$.

Now we consider the estimate only on the particular segment $z = x + iy_0$ with $0 < x < t_1$ in the annulus 
$\{z \in D:0 < d(z, \partial D) < t_1\}$. On this segment,
we have that 
\begin{equation}\label{relation3}
({\rm Re}z) |Pf(z)| = d(z, \partial D) |Pf(z)| \leq \rho_D^{-1}(z) |Pf(z)| \leq \hat\beta_D(t_1). 
\end{equation}
Since $iy_0 \in i\mathbb R$ is taken arbitrarily, combining \eqref{relation1}, \eqref{relation2}, and \eqref{relation3} together, 
we obtain
$$
\beta(t_1^{2}) \leq 8 \sigma (\lambda^{-1}(t_1)) + 16 t_1.
$$
This tends to $0$ as $t_1 \to 0$.  
\end{proof}
 
\begin{remark}
To the best of our knowledge, the equivalence of the conditions in Theorem \ref{PS}
is known only under the extra assumption that $f$ can be quasiconformally extended to $\overline{\mathbb C}$ 
(see \cite[Theorem 6.2]{Sh}, \cite[Theorem 5.1]{WM-1}).
By a similar proof or by an application,
we see the statement of Theorem \ref{PS} is true also for $\mathbb D$. Namely, 
$\hat \beta(t) \to 0$ if and only if $\hat \sigma(t) \to 0$ as $t \to 0$ without any extra assumption on
a univalent analytic function $g$ on $\mathbb D$. This improves the result in \cite{Be87}.
\end{remark}

\section{ Quasiconformal extensions by means of Schwarzian derivative (Proof of Theorem \ref{main1})}\label{mainproof}

In this section, we prove Theorem \ref{main1}. 
For this argument, the following chordal analogue of Becker's result \cite{Be72} mentioned in the introduction 
plays an important role, which was obtained in Gumenyuk and Hotta \cite[Theorem 3.5]{GH}.

\begin{theorem}\label{chord}
Let $(f_t)$ be a chordal Loewner chain over the interval $0 \leq t < \tau$ in $\mathbb H$ with 
Herglotz function $p$. Suppose that there exists a constant $k \in [0, 1)$ such that 
\begin{equation}
\label{kdisk}
p(z, t) \in  
U(k)= \left\{w \in \mathbb H: \; \left| \frac{w - 1}{w + 1}\right| \leq k \right\}
\end{equation}
for all $z \in \mathbb H$ and almost every $0 \leq t < \tau$. 
Then, 
\begin{enumerate}
\item[(i)]
$f_t$ has a continuous extension to $i\mathbb R$ for all $0 \leq t < \tau$;  
\item[(ii)]
$f_t$ can be $k$-quasiconformally extended to 
$\{z  \in \mathbb C:  0 \leq - {\rm Re}\, z < \tau -t\}$ by setting $f_t(z) = f_{t- {\rm Re} z}(i\,{\rm Im} z)$
for all $0 \leq t < \tau$;
\item[(iii)]
all elements of the evolution family $(\varphi_{s,t})$ associated with $p$ are  
$k$-quasiconformally extendible to $\overline{\mathbb C}$ with $\varphi_{s,t}(\infty) = \infty$ for all $s \geq 0$ and $s \leq t < \tau$. 
\end{enumerate}
\end{theorem}

We remark that the original statement of \cite[Theorem 3.5]{GH} was concerning the case of
$\tau=+\infty$, namely, the time variable $t$ extends to $+\infty$ as usual in the Loewner theory.
By examining its proof, however, we see that the above form of the statement is valid for any constant $\tau$ 
in the interval $(0, +\infty]$.

For the construction of the quasiconformal extension $\hat h$ of $h$ in Theorem \ref{main1},  a chordal Loewner chain $(h_t)$ over the interval $[0, \tau)$ 
will be formed  so that $h_0(z) = h(z)$ and the boundary values of $h_t(z)$ yield the extension $\hat h$, 
as is shown in (ii) of Theorem \ref{chord}. 
In the next lemma, we show that a canonical family of analytic functions $(h_t)$ can be constructed as in \cite[Proposition 5.5]{GH} and
it satisfies the chordal Loewner--Kufarev PDE with an appropriate Herglotz function.

\begin{lemma}\label{newlemma}
Let $h$ be univalent and analytic in $\mathbb H$ with $\lim_{z \to \infty}h(z) = \infty$.
Let $h$ satisfy
\begin{equation}\tag{\ref{Slimsup2}}
\lim_{{\rm Re} z\to 0^+} \frac{1}{2}(2{\rm Re}z)^2 \left|Sh(z)\right| = 0.
\end{equation}
Then, for any $0 < k <1$, there is a positive constant $\tau_0>0$ such that
functions on $\mathbb H$ defined by
\begin{equation}\label{family}
h_t(z) = h(z+t) - \frac{2t h'(z+t)}{1+tPh(z+t)}
\end{equation}
for $0 \leq t \leq \tau_0$ are analytic, and this family $(h_t)$ satisfies $h_0=h$ and
the chordal Loewner--Kufarev PDE 
\begin{equation}\label{PDE5}
\dot{h}_t(z) = -p(z, t) h'_t(z) 
\end{equation}
with Herglotz function $p(z,t)$ satisfying \eqref{kdisk}.
\end{lemma}

\begin{proof}
In view of Theorem \ref{PS}, \eqref{Slimsup2} implies that 
\begin{equation}\tag{\ref{Plimsup2}}
\lim_{{\rm Re} z\to 0^+} (2{\rm Re} z) \left|Ph(z)\right| = 0, 
\end{equation}
that is, for any $0 < k < 1$, there exists a positive constant $t_0$ such that 
$$
(2{\rm Re} z) \left|Ph(z)\right| \leq k 
$$
for all $z \in \mathbb H$ with $0 < {\rm Re}\, z \leq t_0$. It follows that 
$$
t |Ph(iy+t)| = {\rm Re}(iy+t) |Ph(iy+t)| \leq k/2
$$
for all $iy \in i\mathbb R$ and all $0 < t \leq t_0$. A generalized maximal principle (see \cite[p.56]{Ga}) implies  
that $t Ph(z+t)$, as an analytic function of the variable $z$, is bounded in $\mathbb H$ such that  
\begin{equation}\label{denom1}
t |Ph(z+t)| \leq k/2, \qquad 0 \leq t \leq t_0. 
\end{equation}

Similarly, according to \eqref{Slimsup2}, there exists a positive constant $\tau_0  \;(< t_0)$ such that  
\begin{equation*}\label{Sk}
 \frac{1}{2}(2{\rm Re} z)^2 \left|Sh(z)\right| \leq k, 
\end{equation*}
for all $z \in \mathbb H$ with $0 < {\rm Re}\, z \leq \tau_0$. 
From this it follows that 
$$
2t^2 |Sh(iy+t)| = \frac{1}{2} \left(2{\rm Re} (iy+t)\right)^2 \left|Sh(iy + t)\right| \leq k
$$ 
for all $iy \in i\mathbb R$ and all $0 < t \leq \tau_0$. 
The generalized maximal principle again implies that 
\begin{equation}\label{denom2}
2t^2 |Sh(z+t)| \leq k, \qquad 0 \leq t \leq \tau_0 
\end{equation}
for all $z \in \mathbb H$.

We consider the family of functions given by \eqref{family}.
This is well-defined to be analytic in $\mathbb H$ for $0 \leq t \leq t_0$ by \eqref{denom1}.  
Obviously, $h_0(z) = h(z)$. By computation, we have
\begin{equation}\label{hdot}
\dot{h}_t(z) := \frac{d h_t(z)}{d t} = -h'(z + t)\frac{1- 2t^2 Sh(z+t)}{\left( 1 + tPh(z + t) \right)^2},
\end{equation}
and
\begin{equation}\label{ht}
h'_t(z) := \frac{d h_t(z)}{d z} =h'(z + t)\frac{1+ 2t^2 Sh(z+t)}{\left( 1 + tPh(z + t) \right)^2}. 
\end{equation}
From these expressions \eqref{hdot} and \eqref{ht}, it follows that $(h_t)$ satisfies the chordal Loewner--Kufarev PDE
\eqref{PDE5}
with Herglotz function 
\begin{equation}\label{Sherg}
p(z, t) = \frac{1- 2t^2 Sh(z+t)}{1 + 2t^2 Sh(z+t)}.
\end{equation}
Moreover, by \eqref{denom2} we can deduce that
\begin{equation}\label{Spk}
\left| \frac{p(z, t) - 1}{p(z, t) + 1} \right| = 2t^2 |Sh(z+t)| \leq k <1
\end{equation}
for all $z \in \mathbb H$ and all $0 \leq t \leq \tau_0$. Namely, $p(z, t) \in U(k)$. 
\end{proof}

In the rest of this section, we will focus our attention on showing the univalence of the analytic function
$h_t$ in $\mathbb H$ defined by \eqref{family} 
for each $t$ in some interval. 
The arguments will be separated into the cases of assumptions $\mathbf{(A)}$ and $\mathbf{(B)}$.

\begin{theorem}\label{univalentA}
Let $h$ be univalent and analytic in $\mathbb H$ with $\lim_{z \to \infty}h(z) = \infty$ and let $h(\mathbb H)$ be a Jordan domain. Let $h$ satisfy
\begin{equation}\tag{\ref{Slimsup2}}
\lim_{{\rm Re} z\to 0^+} \frac{1}{2}(2{\rm Re}z)^2 \left|Sh(z)\right| = 0.
\end{equation}
Suppose $\mathbf{ (A)}$:  both $h$ and $h^{-1}$ are locally uniformly continuous in a neighborhood of $\infty$. 
Then, there is a positive constant $\tau \leq \tau_0$ such that the analytic function $h_t$ on $\mathbb H$
defined by \eqref{family} is univalent for all $0 \leq t \leq \tau$.
\end{theorem}

\proof
The proof contains several claims. For all claims as well as other explanations, the circumstances of the theorem are always assumed.
By Lemma \ref{newlemma}, the family $(h_t)$ defined by \eqref{family} for $0 \leq t \leq \tau_0$
satisfies the chordal Loewner--Kufarev PDE \eqref{PDE5} with the Herglotz function $p(z,t)$ in \eqref{kdisk}.
In particular, $|p(z, t)| \leq K$ for $K := (1+k)/(1-k)$. 

Let $(\varphi_{s,t})$ be the evolution family associated with this Herglotz function $p$ for all $0 \leq s \leq t \leq \tau_0$. Namely, $(\varphi_{s,t})$ satisfies the chordal Loewner--Kufarev ODE:
\begin{equation}\label{p2}
\dot{\varphi}_{s,t}(z) = p(\varphi_{s,t}(z), t), \quad 0 \leq s \leq t \leq \tau_0, \quad \varphi_{s,s}(z) = z.
\end{equation}
By Theorem \ref{chord},  $\varphi_{s,t}$ is $k$-quasiconformally extendible to $\overline{\mathbb C}$ with $\varphi_{s,t}(\infty) = \infty$
for some $0<k <1$. Combining \eqref{PDE5} and \eqref{p2}, we obtain that for any fixed $z \in \mathbb H$, 
\begin{equation*}
\begin{split}
\frac{d}{dt}(h_{t} \circ \varphi_{s,t}(z)) &= h_t'(\varphi_{s,t}(z)) \dot\varphi_{s,t}(z) + \dot h_t(\varphi_{s,t}(z))\\
& = h_t'(\varphi_{s,t}(z)) \dot\varphi_{s,t}(z) - p(\varphi_{s,t}(z), t) h_t'(\varphi_{s,t}(z))\\
& = h_t'(\varphi_{s,t}(z)) (\dot\varphi_{s,t}(z) - p(\varphi_{s,t}(z), t)) = 0.
\end{split}
\end{equation*}
Thus, $h_{t} \circ \varphi_{s,t}(z)$ does not depend on $t$. By taking $t = s$ we see that $h_{t} \circ \varphi_{s,t}(z) = h_s(z)$.  
The equations $\varphi_{0,t} = \varphi_{s,t}\circ\varphi_{0,s}$ and $h_s = h_{t} \circ \varphi_{s,t}$ will be frequently used below.

We note that the univalent analytic function $h$ on $\mathbb H$ 
has a continuous and injective extension to $i\hat{\mathbb R}=i\mathbb R \cup \{\infty\}$ with $h(\infty)=\infty$
by the Carath\'eodory theorem (see \cite[p.18]{Pom}) since $h(\mathbb H)$ is a Jordan domain.

\begin{claim}\label{S1}
The value $h_t(z)$ converges to $h(z_0)$ as $z \in \mathbb H$ tends to $z_0 \in \mathbb H \cup i\hat{\mathbb R}$
and $t \in [0, \tau_0]$ tends to $0$. 
\end{claim}

\begin{proof}
For any $z \in \mathbb H$ and $0 \leq t \leq t' \leq \tau_0$, 
\begin{equation}\label{basic}
\begin{split}
|\varphi_{t,t'}(z) - z| & = |\varphi_{t,t'}(z) - \varphi_{t,t}(z)|\\
& = \left|\int_{t}^{t'} \dot{\varphi}_{t,s}(z) ds\right| 
 = \left|\int_{t}^{t'} p(\varphi_{t,s}(z), s) ds\right|\\
& \leq \int_{t}^{t'} |p(\varphi_{t,s}(z), s)| ds 
 \leq K(t' - t)  
\end{split}
\end{equation}
is satisfied. This is also true for $z \in i\mathbb R$ since
each $\varphi_{t, t'}$ has the continuous extension to $i\mathbb R$. 
In particular, 
\begin{equation}\label{closetoid}
|\varphi_{0,t}(z)-z| \leq Kt
\end{equation}
for any $z \in \mathbb H \cup i\mathbb R$ and $0 \leq t \leq \tau_0$.
Let $\Omega(M):=\{z \in \mathbb H :|z| \leq M\}$ for any $M>0$. Then, 
\begin{equation}\label{bounded}
|\varphi_{t,\tau_0}(z)| \leq |\varphi_{t,\tau_0}(z) - z| + |z| \leq K\tau_0 + M 
\end{equation}
for any $z \in \Omega(M)$.
This shows that the family $\{\varphi_{t, \tau_0}\}_{t \in [0,\tau_0]}$ is uniformly bounded on 
the compact set $\overline \Omega(M)=\{z \in \mathbb H \cup i\mathbb R:|z| \leq M\}$ for any $M>0$. 

If the family of $k$-quasiconformal mappings $\varphi_{t,\tau_0}:\mathbb C \to \mathbb C$ is uniformly bounded on 
$\overline \Omega(M+K\tau_0)$, then it is also equi-continuous on $\overline \Omega(M+K\tau_0)$ (see \cite[Theorem II.4.1]{LV}), 
and thus uniformly H\"older continuous with exponent $1/K$ and with multiplicative constant $C_1>0$
(see \cite[Theorem II.4.3]{LV}). From this and \eqref{closetoid}, it follows that 
\begin{equation}\label{Holder}
\begin{split}
|\varphi_{t,\tau_0}(z) - \varphi_{0,\tau_0}(z)| &= |\varphi_{t,\tau_0}(z) - \varphi_{t,\tau_0}\circ\varphi_{0,t}(z)|\\
& \leq C_1 |z - \varphi_{0,t}(z)|^{1/K} 
\leq C_1(Kt)^{1/K} 
\end{split}
\end{equation}
for any $z \in \Omega(M)$.

Since $p(z,t)\in U(k)$ for all $z \in \mathbb H$ and all $0 \leq t \leq \tau_0$, there exists a positive constant $C_2>0$ such that
${\rm Re}\,p(z,t) \geq C_2$. It follows that
\begin{equation*}
\begin{split}
{\rm Re}\,(\varphi_{0,t}(z)) - {\rm Re}\,z &= \int_0^{t}  {\rm Re}\,(\dot{\varphi}_{0,s}(z)) ds\\
&= \int_0^{t}  {\rm Re}\,(p(\varphi_{0,s}(z), s)) ds \geq C_2t.
\end{split}
\end{equation*}
This implies that $\varphi_{0,\tau_0}(\mathbb H)$ is contained in $\{z \in \mathbb H : {\rm Re}\,z >C_2\tau_0\}$.
Combining this with \eqref{bounded} and \eqref{Holder},
we can take a compact convex subset $W$ in $\mathbb H$ such that $\varphi_{t,\tau_0}(z)$ and $\varphi_{0,\tau_0}(z)$
are contained in $W$ for all $z \in \Omega(M)$ and for all sufficiently small $t$.

Let $C_3:=\max_{w \in W}|h_{\tau_0}'(w)|<+\infty$.
Then, for any two points $w_1$, $w_2 \in W$, 
\begin{equation*}
\begin{split}
|h_{\tau_0}(w_1) - h_{\tau_0}(w_2)| & =\left |\int_0^1(h_{\tau_0}((1-t)w_1 + tw_2))_t'  dt\right|\\
& = |w_2 - w_1| \left|\int_0^1h_{\tau_0}'((1-t)w_1 + tw_2) dt\right| 
\leq 
C_3  |w_2 - w_1|, 
\end{split}
\end{equation*}
from which we deduce that
\begin{equation}\label{final}
\begin{split}
|h_{t}(z) - h(z_0)| & \leq |h_{t}(z) - h(z)|  + |h(z) - h(z_0)|\\
& = |h_{\tau_0}\circ\varphi_{t,\tau_0}(z) - h_{\tau_0}\circ\varphi_{0,\tau_0}(z)| + |h(z) - h(z_0)|\\
& \leq C_3 |\varphi_{t,\tau_0}(z) - \varphi_{0,\tau_0}(z)| + |h(z) - h(z_0)|
\end{split}
\end{equation}
if $z \in \Omega(M)$.
In the case that $z_0 \neq \infty$, by choosing $\Omega(M)$ according to $z_0$,
we see that $h_t(z) \to h(z_0)$ as $z \to z_0$
and $t \to 0$ by \eqref{Holder} and \eqref{final}.

Finally, we assume that $z_0 = \infty$.
By the equation $h_{\tau_0}\circ\varphi_{0,\tau_0} = h$, we have 
$\lim_{z \to \infty} h_{\tau_0}(z) = \lim_{z \to \infty}h\circ \varphi_{0,\tau_0}^{-1}(z) = \infty$. Moreover,
\begin{equation*}
\begin{split}
|\varphi_{t, \tau_0}(z)|  
& \geq 
|z| - |\varphi_{t, \tau_0}(z) - \varphi_{t,t}(z)|\\
&= |z| - \left|\int_{t}^{\tau_0} \dot{\varphi}_{t,s}(z) ds \right|\\
& \geq |z| - \int_{t}^{\tau_0} |p(\varphi_{t,s}(z), s)| ds \geq |z| - K\tau_0.
\end{split}
\end{equation*}
Hence,
$h_{t}(z) = h_{\tau_0}\circ\varphi_{t,\tau_0}(z) \to \infty \, (= h(\infty))$ 
as $z \to \infty$
(independent of $t$).
\end{proof}

\begin{claim}\label{S2}
There exist positive constants $\tau_1 \;(\leq \tau_0)$ and $r$ such that for any $0 \leq t \leq \tau_1$,  $h_t(z)$ is univalent on each square 
$$Q_{2r} := (0, 2r)\times(y_0-r, y_0 + r) \subset \mathbb H,$$
where $y_0 \in \mathbb R$ can be chosen arbitrarily. 
\end{claim}

Remark that the square $Q_{2r}$, as a quasidisk, satisfies the Schwarzian univalence criterion 
(see \cite[Corollary 5]{GO}, \cite[p.126]{Le}). 
Precisely, 
if a function $g$ is  locally univalent and analytic in $Q_{2r}$ and if 
$$
\sup_{z \in Q_{2r}} |Sg(z)|\rho_{Q_{2r}}^{-2}(z) \leq 1/2,
$$
then $g$ is univalent in $Q_{2r}$. Moreover, by \eqref{sup}, there exists a positive constant $a$ such that if 
$$
\sup_{z \in Q_{2r}} |Pg(z)|\rho_{Q_{2r}}^{-1}(z) \leq a,
$$
then $g$ is univalent in $Q_{2r}$. 

We also need the following result which can be deduced easily from the proof of \cite[Lemma 6.3]{Sh}.

\begin{proposition}\label{alpha1}
Let $\psi$ be an analytic function of $z=x+iy$ in $\mathbb H$ satisfying a condition $\lim_{x \to +\infty} \psi(x+iy) = 0$ 
uniformly for $y \in \mathbb R$. Then,  for any constant $\alpha > 0$, 
$$
\sup_{z \in \mathbb H} |\psi(z)|x^{\alpha}  < \infty \quad \Longleftrightarrow \quad
\sup_{z \in \mathbb H} |\psi'(z)|x^{\alpha+1} < \infty, 
$$
and 
both terms are comparable with comparison constants depending only on $\alpha$. Moreover, 
$$
\lim_{x \to 0^+} |\psi(z)| x^{\alpha} = 0 \quad \Longleftrightarrow \quad \lim_{x \to 0^+} |\psi'(z)| x^{\alpha+1} = 0.
$$
\end{proposition}

\begin{proof}[Proof of Claim \ref{S2}]
By \eqref{denom2} and \eqref{ht}, 
we see that  $h'_t(z) \neq 0$ for all $z \in \mathbb H$ and all $0 \leq t \leq \tau_0$. 
A direct computation yields that
\begin{equation*}
\begin{split}
Ph_t(z) &:= \frac{h_t''(z)}{h_t'(z)} = (\log h_t'(z))' \\
&= Ph(z+t) + \frac{2t^2 \left(Sh(z+t) \right)'}{1 + 2t^2 Sh(z+t)} - \frac{2t\left(Ph(z+t) \right)'}{1 + t Ph(z+t)}. 
\end{split}
\end{equation*}
From this we have 
\begin{equation*}\label{zt}
\begin{split}
(2{\rm Re} z) \left|Ph_t(z)\right| & \leq  (2{\rm Re}(z+t)) |Ph(z+t)| \\
&\quad + \frac{(2{\rm Re}(z+t))^3 |(Sh(z+t))'|}{1 - (2{\rm Re}(z+t))^2 |Sh(z+t)|} + \frac{(2{\rm Re}(z+t))^2 |(Ph(z+t))'|}{1 - (2{\rm Re}(z+t)) |Ph(z+t)|}.
\end{split}
\end{equation*}
Then, applying \eqref{Slimsup2} and \eqref{Plimsup2}, 
we conclude by Proposition \ref{alpha1} that there exists a positive constant $\tau_1 \;(\leq \tau_0)$ such that 
\begin{equation*}\label{c0}
(2{\rm Re} z) \left|Ph_t(z)\right| \leq a
\end{equation*}
for any $z \in \mathbb H$ and $t \geq 0$ with $0 < {\rm Re}(z+t) \leq 3\tau_1$. In particular, it holds for all $z \in \mathbb H$ with $0 < {\rm Re}\, z \leq 2r$ and for all $0 \leq t \leq \tau_1$ by taking $r = \tau_1$.  
We fix this constant $r>0$.

By the monotone principle of hyperbolic densities with respect to the inclusion of
domains (see \cite[p.6]{Le}), we have 
$$
\rho_{Q_{2r}}^{-1}(z) \left| Ph_t(z) \right| \leq (2 {\rm Re} z) \left| Ph_t(z) \right| \leq a 
$$
for all $z \in Q_{2r}$, 
which implies that $h_t$ is univalent in each $Q_{2r}$ for any $0 \leq t \leq \tau_1$. 
\end{proof}

\begin{claim}\label{S3}
There exists a positive constant $\tau_2 \;(\leq \tau_0)$ such that for each $0 \leq t \leq \tau_2$, 
$h_t$ is univalent in the half-plane
$
\mathbb H_{r} := \{z \in \mathbb H: {\rm Re} z > r\},
$
where $r$ is the constant chosen in Claim \ref{S2}.
\end{claim}
\begin{proof}
Since $h$ and $\varphi_{0,t}$ are univalent in $\mathbb H$ for any $0 \leq t \leq \tau_0$ and since $h_t\circ\varphi_{0,t} = h$, 
we see that $h_t$ is univalent in the domain $\varphi_{0,t}(\mathbb H)$. Here, \eqref{basic} in particular
implies that $\varphi_{0,t}$ converges to ${\rm id}_{\mathbb H}$
uniformly with respect to $z \in \mathbb H$ as $t \to 0$. Then, there exists a positive constant $\tau_2 \;(\leq \tau_0)$ such that 
$\mathbb H_{r} \subset \varphi_{0,t}(\mathbb H)$ for each $0 \leq t \leq \tau_2$, and thus $h_t$ is univalent in $\mathbb H_{r}$. 
\end{proof}

Let us recall the well-known Koebe distortion theorem (see \cite[Theorem 1.3]{Pom}). It says that if $g$ is analytic and univalent in $\mathbb D$ then for all $z \in \mathbb D$ it holds that
$$
|g'(0)|\frac{|z|}{(1+|z|)^2} \leq |g(z) - g(0)| \leq |g'(0)| \frac{|z|}{(1 - |z|)^2}.
$$
Applying it to an analytic and univalent function $g$ on a closed disk $\overline D(z_0+t, t)$
with center $z_0+t$ and radius $t$ by translation and dilation, we have that for all $z \in \overline{\mathbb D}$, 
$$
t|g'(z_0+t)|\frac{|z|}{(1+|z|)^2} \leq |g(tz+(z_0+t)) - g(z_0+t)| \leq t|g'(z_0+t)| \frac{|z|}{(1 - |z|)^2}.
$$
In particular, the left inequality for $z = -1$ yields that
\begin{equation}\label{distortion}
|g(z_0) - g(z_0 + t)| \geq \frac{t}{4}|g'(z_0+t)|.
\end{equation}

\begin{claim}\label{S4}
Suppose that assumption $\mathbf{(A)}$ holds. Namely, $h$ and $h^{-1}$ are uniformly continuous in 
$\Omega_M := \{z \in \mathbb H: |z| > M \}$ for some $M > 0$. Then, there exist positive constants $\tau_3\; (< \tau_0)$ and
$R > M$
such that if $z_1, z_2 \in \Omega_R$ with $|z_1 - z_2| \geq r$, then $h_t(z_1) \neq h_t(z_2)$ for all $0 \leq t \leq \tau_3$,
where $r$ is the constant chosen in Claim \ref{S2}.
\end{claim}

\begin{proof}
Let $\zeta_1 = h(z_1)$ and $\zeta_2 = h(z_2)$. As $h(\infty) = \infty$, there is a constant $R > M$ such that
if $z_1, z_2 \in \Omega_{R}$ then $\zeta_1, \zeta_2 \in \Omega_M$.
Since $h^{-1}$ is uniformly continuous in $\Omega_M$, for the given constant $r$, there exists some $\delta > 0$ such that
$$
|\zeta_1 - \zeta_2| < \delta \quad \Longrightarrow \quad |z_1 - z_2| < r,
$$
or equivalently,
\begin{equation}\label{unifcont1}
|z_1 - z_2| \geq r \quad \Longrightarrow \quad |\zeta_1 - \zeta_2| \geq \delta 
\end{equation}
for $z_1, z_2 \in \Omega_{R}$.
Moreover, since $h$ is uniformly continuous in $\Omega_{R} \subset \Omega_M$,
for this $\delta$, there exists a positive constant $\tau_3\; (< \tau_0)$ such that
\begin{equation}\label{unifcont2}
|z_1 - z_2| \leq \tau_3 \quad \Longrightarrow \quad |\zeta_1 - \zeta_2| \leq \delta/64
\end{equation}
for $z_1, z_2 \in \Omega_{R}$.

Assume that $z_1, z_2 \in \Omega_R$ with $|z_1 - z_2| \geq r$ and $0 \leq t \leq \tau_3$. By using \eqref{denom1}, 
\eqref{distortion}, \eqref{unifcont1}, and \eqref{unifcont2} in this order, we have
\begin{equation*}
\begin{split}
|h_t(z_1) - h_t(z_2)| & = \left|h(z_1 + t) - h(z_2 + t) -\frac{2t h'(z_1 + t)}{1 + t Ph(z_1 + t)} + \frac{2t h'(z_2 + t)}{1 + t Ph(z_2 + t)}\right|\\
& \geq |h(z_1 + t) - h(z_2 + t)| -4t |h'(z_1 + t)| -4t |h'(z_2 + t)|\\
& \geq |h(z_1 + t) - h(z_2 + t)| - 16 |h(z_1 + t) - h(z_1)| - 16 |h(z_2 + t) - h(z_2)|\\
& \geq \delta - \delta/4 - \delta/4 >0.
\end{split}
\end{equation*}
This proves that $h_t(z_1) \neq h_t(z_2)$.
\end{proof}

\begin{claim}\label{S6}
There exists a positive constant $\tau_4 \;(\leq \tau_0)$ such that if $|z_1 - z_2| \geq r$ and if either $|z_1| \leq R$ or $|z_2| \leq R$, then $h_t(z_1) \neq h_t(z_2)$ for all $0 \leq t \leq \tau_4$. Here, $r$ is the constant chosen in Claim \ref{S2} and
$R$ is the constant chosen in Claim \ref{S4}.
\end{claim}

\begin{proof}
Suppose that there is no such $\tau_4$. 
Then, there exist a sequence $\{t_n\}$ with $t_n \to 0$ as $n \to \infty$, and sequences $\{z_{1n}\}$, $\{z_{2n}\}$ in $\mathbb H$ 
satisfying $|z_{1n} - z_{2n}| \geq r$ and either
$|z_{1n}| \leq R$ or $|z_{2n}| \leq R$, but $h_{t_n}(z_{1n}) = h_{t_n}(z_{2n})$ for all $n$.
We may assume that $\{z_{1n}\}$ and $\{z_{2n}\}$ converge to some points $z_1$ and $z_2$ in $\mathbb H \cup i\hat{\mathbb R}$ 
respectively by passing to subsequences.

By Claim \ref{S1}, $h_{t_n}(z_{1n}) \to h(z_1)$ and $h_{t_n}(z_{2n}) \to h(z_2)$ as $n \to \infty$. Since 
$h_{t_n}(z_{1n}) = h_{t_n}(z_{2n})$, we have 
$h(z_1) = h(z_2)$. However, both the conditions that $|z_{1n} - z_{2n}| \geq r$ and that either
$|z_{1n}| \leq R$ or $|z_{2n}| \leq R$ imply that
$z_1 \neq z_2$. This contradicts the fact that $h$ is injective in $\mathbb H\cup i\hat{\mathbb R}$. 
\end{proof}

We will finish the proof.
Set $\tau := \min\{ \tau_1, \tau_2, \tau_3, \tau_4\}>0$.
Let $r$ be the constant chosen in Claim \ref{S2}, and
$R$ the constant chosen in Claim \ref{S4}.
We show that $h_t$ is univalent in $\mathbb H$ for
$0 \leq t \leq \tau$.
For any distinct points $z_1$ and $z_2$ in $\mathbb H$,
one of the following three cases occurs:
\begin{itemize} 
\item
If $|z_1 - z_2| < r$, then $h_t(z_1) \neq h_t(z_2)$ by Claims \ref{S2} and \ref{S3}.
\item
If $|z_1 - z_2| \geq r$ and if both $|z_1| > R$ and $|z_2| > R$, then $h_t(z_1) \neq h_t(z_2)$ by Claim \ref{S4}.
\item 
If $|z_1 - z_2| \geq r$ and if either $|z_1| \leq R$ or $|z_2| \leq R$, then $h_t(z_1) \neq h_t(z_2)$ by Claim \ref{S6}.
\end{itemize}
In any case, we have $h_t(z_1) \neq h_t(z_2)$ and hence $h_t$ is injective in $\mathbb H$.
This completes the proof of Theorem \ref{univalentA}.
\qed
\medskip

Next, we prove the univalence of $h_t$ under assumption $\mathbf{(B)}$.

\begin{theorem}\label{univalentB}
Let $h$ be univalent and analytic in $\mathbb H$ with $\lim_{z \to \infty}h(z) = \infty$ and let $h(\mathbb H)$ be a Jordan domain. Let $h$ satisfy \eqref{Slimsup2}.
Suppose $\mathbf{(B)}$: $h$ is locally quasiconformally extendible to a neighborhood of $\infty$. 
Then, there is a positive constant $\tau \leq \tau_0$ such that the analytic function $h_t$ on $\mathbb H$
defined by \eqref{family} is univalent for all $0 \leq t \leq \tau$.
\end{theorem}

\begin{proof}
In the proof of Theorem \ref{univalentA}, assumption $\mathbf{(A)}$ is only used in Claim \ref{S4}; 
the other claims are applicable also to 
the proof of Theorem \ref{univalentB}. Hence, it is enough to show the following parallel result, Claim \ref{S5}, to Claim \ref{S4} and 
reset $\tau := \min\{ \tau_1, \tau_2, \tau_4, \tau_5\}$ and $R$, where $\tau_5$ and $R$ are determined there. 
\end{proof}

\begin{claim}\label{S5}
Suppose that assumption $\mathbf{(B)}$ holds. Namely, $h$ has a quasiconformal extension to  $D_{M} := \{z \in \mathbb C:|z| > M \}$ for some $M > 0$. Then, there exist positive constant $\tau_5 \; (\leq \tau_0)$ and $R > M$
such that if $z_1, z_2 \in \Omega_R$ with $|z_1 - z_2| \geq r$, then $h_t(z_1) \neq h_t(z_2)$ for all $0 \leq t \leq \tau_5$,
where $r$ is the constant chosen in Claim \ref{S2}.
\end{claim}

The key concept we use to prove Claim \ref{S5} is that of quasisymmetry which was introduced by Beurling and Ahlfors 
\cite{BA} on the real line and formulated for general metric spaces by Tukia and V\"ais\"al\"a \cite{TV}. 
For our purpose we only consider it for the complex plane. 
We refer to \cite[Definition 3.2.1]{AIM}.

\begin{definition}
Let $D \subset \mathbb C$ be an open subset and $f:D \to \mathbb C$ an orientation-preserving mapping. Let $\eta:[0,\, +\infty) \to [0,\, +\infty)$ be an increasing homeomorphism with $\lim_{t \to +\infty}\eta(t)=+\infty$. We say that $f$ is $\eta$-quasisymmetric if for each triple $z_0$, $z_1$, $z_2 \in D$ we have
$$
\frac{|f(z_0) - f(z_1)|}{|f(z_0) - f(z_2)|} \leq \eta\left( \frac{|z_0 - z_1|}{|z_0 - z_2|} \right).
$$
\end{definition}

If $f: \mathbb C \to \mathbb C$ is a $k$-quasiconformal homeomorphism of $\mathbb C$, then $f$ is $\eta$-quasisymmetric where $\eta$ depends only on $k$ (see \cite[Theorem 3.5.3]{AIM}); conversely, if $f: D \to \mathbb C$ is an $\eta$-quasisymmetric mapping on a domain $D$ then $f$ is quasiconformal (see \cite[Theorem 3.4.1]{AIM}).

\begin{proof}[Proof of Claim \ref{S5}]
Suppose that $h$ has a quasiconformal extension to $D_{M}$, which we still denote by $h$. Let $R=M+1$.
By an extension theorem (see \cite[Theorem II.1.8]{LV}),    
there exists a quasiconformal homeomorphism $F$ of the whole plane $\mathbb C$ that coincides with $h$ in $D_{R}$. In particular, it coincides with $h$ in $\Omega_{R}$. 
The inverse $F^{-1}$ of $F$ is also a quasiconformal homeomorphism of $\mathbb C$, 
and then it is a quasisymmetric homeomorphism of $\mathbb C$. Thus, 
there exists an increasing homeomorphism $\eta: [0, +\infty) \to [0, +\infty)$ such that
$$
\frac{|F^{-1}(w_0) - F^{-1}(w_1)|}{|F^{-1}(w_0) - F^{-1}(w_2)|}  \leq \eta\left( \frac{|w_0 - w_1|}{|w_0 - w_2|} \right)
$$
for any $w_0, w_1, w_2 \in \mathbb C $. 
We take the inverse $\lambda := \eta^{-1}$. 

For any $z_1, z_2 \in \Omega_R$ with $|z_1 - z_2| \geq r$ and any $t>0$, we have
\begin{equation}\label{qs}
\begin{split}
\frac{|h(z_1+t) - h(z_2 + t)|}{|h(z_1) - h(z_1 + t)|} & = \frac{|F(z_1+t) - F(z_2 + t)|}{|F(z_1) - F(z_1 + t)|}\\
& \geq \lambda\left( \frac{|(z_1 + t) - (z_2 + t)|}{|z_1 - (z_1 + t)|} \right) 
= \lambda\left( \frac{|z_1 - z_2|}{t} \right) \geq \lambda\left( \frac{r}{t} \right). 
\end{split}
\end{equation}
We may assume that $|h'(z_1 + t)| \geq |h'(z_2 + t)|$ by exchanging the roles of $z_1$ and $z_2$.
Then, by using \eqref{denom1}, \eqref{distortion}, and \eqref{qs}, we obtain that  
\begin{equation*}
\begin{split}
|h_t(z_1) - h_t(z_2)| & = \left|h(z_1 + t) - h(z_2 + t) -\frac{2t h'(z_1 + t)}{1+tPh(z_1+t)} + \frac{2t h'(z_2 + t)}{1 + t Ph(z_2 + t)}\right|\\
& \geq |h(z_1 + t) - h(z_2 + t)| -4t |h'(z_1 + t)| -4t |h'(z_2 + t)|\\
&\geq |h(z_1) - h(z_1 + t)|\lambda\left( \frac{r}{t} \right) - 8t |h'(z_1 + t)|\\
&\geq \frac{t}{4} |h'(z_1 + t)|\lambda\left( \frac{r}{t} \right) - 8t |h'(z_1 + t)|
= \frac{t}{4} |h'(z_1 + t)|\left(\lambda\left( \frac{r}{t} \right) - 32\right)
\end{split}
\end{equation*}
for all $0 < t \leq \tau_0$.
It can be seen from the monotonicity of $\lambda$ that there exists a positive constant $\tau_4 \; (\leq \tau_0)$ such that 
$\lambda (r/t) > 32$ for all $0 < t \leq \tau_4$. This shows that
$h_t(z_1) \neq h_t(z_2)$ for all $0 \leq t \leq \tau_4$.
\end{proof}

We have proved Theorems \ref{univalentA} and \ref{univalentB}. Then, Theorem \ref{main1} follows from these theorems combined with
Theorem \ref{chord} and Lemma \ref{newlemma}.

\proof[Proof of Theorem \ref{main1}.]
By Theorems \ref{univalentA} and \ref{univalentB},
we see that the analytic function $h_t$ defined by \eqref{family} is univalent in $\mathbb H$ for each $0 \leq t \leq \tau$. Then by Lemma \ref{newlemma},
we obtain that $(h_t)$ is a chordal Loewner chain over the interval $[0, \tau)$ with the associated Herglotz function $p$ satisfying \eqref{kdisk} (see \cite[Theorem 4.1]{CDMG10}). From Theorem \ref{chord}, the assertion of Theorem \ref{main1} follows.
Indeed, the extension $\hat h$ is defined by
\begin{equation*}\label{boundary}
\hat h (z) = h_{-{\rm Re}z}(i{\rm Im}z), \qquad -\tau < {\rm Re}\, z < 0,
\end{equation*}
for the chordal Loewner chain $(h_t)$ over the interval $[0, \tau)$ given by \eqref{family}.
This yields the explicit formula of this extension in \eqref{extension}. From this, its complex dilatation can be computed directly
as in \eqref{Smu3}.
\qed
\medskip

\section{Quasiconformal extensions by means of pre-Schwarzian derivative }\label{preversion}

We have the following parallel result to Theorems \ref{main1} by replacing the Schwarzian derivative $Sf$ of $f$ with the pre-Schwarzian derivative $Pf$ of $f$.

\begin{theorem}\label{Pmain12}
Let $f$ be univalent and analytic in $\mathbb H$ with $\lim_{z \to \infty}f(z) = \infty$ such that  
$f(\mathbb H)$ is a Jordan domain. Let $f$ satisfy the condition
\begin{equation*}
\lim_{{\rm Re} z\to 0^+} (2{\rm Re} z) \left| Pf(z) \right| = 0.
\end{equation*}
Suppose further that either $\mathbf{(A)}$ or $\mathbf{(B)}$ holds. 
Then, the function 
$$
\hat f(z) = 
\begin{cases}
f(z), & \text {if}\;\; {\rm Re}\, z \geq 0,\\
f(z^*) + (2{\rm Re} z) f'(z^*), \quad (z^* := -\bar z) & \text{if}\;\; -\tau < {\rm Re}\, z < 0
\end{cases}
$$
defines a quasiconformal extension of $f$ over $i\mathbb R$ to 
$\mathbb{H}^{*}_{(0,\tau)} := \{z \in \mathbb C: -\tau < {\rm Re}\, z < 0 \}$ for some $\tau > 0$  
such that its complex dilatation $\mu$ on $\mathbb{H}^{*}_{(0,\tau)}$ has the form
\begin{equation*}
\mu(z) = -(2 {\rm Re} z)  Pf(z^*).
\end{equation*}
\end{theorem}

The proof of Theorem \ref{Pmain12} is completely similar to that of Theorem \ref{main1}. 
We only make mention of the difference. Instead of considering the family $(h_t)$ defined by \eqref{family},
we consider a simpler one 
$$
f_t(z) := f(z+t) - 2t f'(z+t), \qquad z \in \mathbb H, \quad 0 \leq t \leq \tau_0 
$$
for some $\tau_0>0$. It is associated with the Herglotz function 
$$
p(z, t) = \frac{1+2t Pf(z+t)}{1 - 2t Pf(z+t)}
$$
in comparison with \eqref{Sherg}.
This satisfies
\begin{equation*}\label{pk}
\left| \frac{p(z, t) - 1}{p(z, t) + 1} \right| = 2t \left|Pf(z+t)\right| \leq k < 1
\end{equation*}
for all $z \in \mathbb H$ and all $0 \leq t \leq \tau_0$ in comparison with \eqref{Spk}. Then, the corresponding statement to Lemma \ref{newlemma}
is obtained. To prove that there is some positive constant $\tau \;(\leq \tau_0)$ such that $f_t$ is univalent 
for any $0 \leq t \leq \tau$, we repeat the arguments for Theorems \ref{univalentA} and \ref{univalentB}.
Because of dealing with the simpler family $(f_t)$, the argument becomes 
a bit simpler. We omit the details here.

\section{Application to VMO-Teichm\"uller space (proof of Theorem \ref{main3})}\label{app}

Before proceeding to the proof of Theorem \ref{main3}, let us recall some basic definitions on Carleson measures and BMO functions (see \cite[Chapter 6]{Ga}). 

We say that a positive measure $\lambda$ on $\mathbb H$ is a Carleson measure if 
$$
\Vert \lambda \Vert_{c} := \sup_{I \subset i\mathbb R} \frac{\lambda ((0,\, |I|) \times I)}{|I|} < \infty,
$$
where the supremum is taken over all bounded intervals $I$ in $i\mathbb R$. A Carleson measure $\lambda$ is called a vanishing Carleson measure if 
$$
\lim_{|I| \to 0} \frac{\lambda ((0,\, |I|) \times I)}{|I|} = 0. 
$$
We denote by ${\rm CM}(\mathbb H)$ and ${\rm CM}_0(\mathbb H)$ the set of all Carleson measures and vanishing Carleson measures on $\mathbb H$, respectively. ${\rm CM}(\mathbb H^*)$ and ${\rm CM}_0(\mathbb H^*)$ can be defined similarly.

A locally integrable complex-valued function $u$ on $i\mathbb R$ is of  BMO (denoted by $u \in {\rm BMO}(i\mathbb R)$) if 
$$
\Vert u \Vert_{*} := \sup_{I \subset i\mathbb R} \frac{1}{|I|} \int_I |u(z) - u_I| |dz| < \infty, 
$$
where the supremum is taken over all bounded intervals $I$ on $i\mathbb R$ and $u_I$ denotes the integral mean of $u$ over $I$.  Moreover, $u$ is of VMO (denoted by $u \in {\rm VMO}(i\mathbb R)$) if in addition
$$
\lim_{|I| \to 0}\frac{1}{|I|} \int_I |u(z) - u_I| |dz| = 0. 
$$
Let $\rm BMOA(\mathbb H)$ denote the space of all analytic functions $\phi$ on $\mathbb H$ that are Poisson integrals of BMO functions on $i\mathbb R$, and let $\rm VMOA(\mathbb H)$ denote the subspace of $\rm BMOA(\mathbb H)$ whose elements have boundary values in ${\rm VMO}(i\mathbb R)$. By using the similar arguments to the case of the unit disk (see \cite[p.233]{Ga}), we can conclude that an analytic function $\phi$ on $\mathbb H$ belongs to $\rm BMOA(\mathbb H)$ if and only if $\phi$ induces a Carleson measure 
\begin{equation*}
(2{\rm Re}z) |\phi'(z)|^2 dxdy \in {\rm CM}(\mathbb H),
\end{equation*} 
and moreover, 
$\phi$ belongs to $\rm VMOA(\mathbb H)$ if and only if 
\begin{equation*}
(2{\rm Re}z) |\phi'(z)|^2 dxdy \in {\rm CM}_0(\mathbb H).
\end{equation*}

We need the following result from \cite[Proposition 7.4]{Sh}. 
This is conceptually similar to Proposition \ref{alpha1}.

\begin{proposition}\label{alpha2}
Let $\psi$ be an analytic function on $\mathbb H$ such that $\lim_{x \to +\infty} \psi(x+iy) = 0$ uniformly for $y \in \mathbb R$. For $\alpha > 0$ set 
$$\lambda_1 := |\psi(z)|^2 x^{\alpha} dxdy$$
and 
$$\lambda_2 := |\psi'(z)|^2 x^{\alpha+2} dxdy$$
for $z = x + iy \in \mathbb H$.
Then, $\lambda_1 \in {\rm CM}(\mathbb H)$ if and only if $\lambda_2 \in {\rm CM}(\mathbb H)$, and $\Vert \lambda_1 \Vert_{c} \asymp \Vert \lambda_2 \Vert_{c}$ with comparison constants depending only on $\alpha$. Moreover, $\lambda_1 \in {\rm CM}_0(\mathbb H)$ if and only if $\lambda_2 \in {\rm CM}_0(\mathbb H)$. 
\end{proposition}

\begin{proof}[Proof of Theorem \ref{main3}]  
Since $h$ admits a quasiconformal extension to $\overline{\mathbb C}$ with $h(\infty) = \infty$, we have the well-known inequalities (see \cite{Be80-1} for more information): 
\begin{equation}\label{norm6}
(2{\rm Re}z) |Ph(z)|  < 6; \quad (2{\rm Re}z)^2 |Sh(z)|  < 6
\end{equation}
for all $z \in \mathbb H$. 
These in particular yield that 
$Ph(z) \to 0$ and $Sh(z) \to 0$ uniformly for ${\rm Im}\,z \in \mathbb R$ as ${\rm Re}\,z \to +\infty$. 
Since $\log h' \in {\rm VMOA}(\mathbb H)$, we obtain by \cite[Lemma 7.1]{Sh} that 
\begin{equation*}
\lim_{{\rm Re} z\to 0^+} (2{\rm Re} z) \left| Ph(z) \right| = 0, 
\end{equation*}
and then by Theorem \ref{PS} that
\begin{equation*}
\lim_{{\rm Re} z\to 0^+} (2{\rm Re} z)^2 \left| Sh(z) \right| = 0. 
\end{equation*}
Thus, $h$ satisfies the condition of Corollary \ref{main2}. In addition,
the Loewner chain $(h_t)$ over some interval $[0, \tau)$ 
produced by $h$ in the form of \eqref{family}  satisfies the condition of Theorem \ref{chord}.

It follows from the equality $Sh = (Ph)' - (Ph)^2/2$ and \eqref{norm6} that
\begin{equation}\label{SC}
\begin{split}
(2{\rm Re}z)^3 |Sh(z)|^2 &\leq 2 (2{\rm Re}z)^3 |(Ph(z))'|^2 + 1/2 \left((2{\rm Re}z) |Ph(z)|\right)^2 (2{\rm Re}z) |Ph(z)|^2\\
& \leq 2 (2{\rm Re}z)^3 |(Ph(z))'|^2 + 18 (2{\rm Re}z) |Ph(z)|^2.
\end{split}
\end{equation}
Since $\log h' \in {\rm VMOA}(\mathbb H)$, or equivalently, 
\begin{equation}\label{PhCM}
(2{\rm Re}z) |Ph(z)|^2 dxdy \in {\rm CM}_0(\mathbb H),
\end{equation} 
we conclude by Proposition \ref{alpha2} and \eqref{SC} that 
\begin{equation}\label{ShCM}
(2{\rm Re}z)^3 |Sh(z)|^2 dxdy \in {\rm CM}_0(\mathbb H).
\end{equation}

We choose any $0 < t < \tau$ and fix it. 
It follows from \eqref{ht} that 
\begin{equation*}
Ph_t(z) = \left( \log h_t'(z) \right)' = Ph(z+t) + \frac{2t^2 (Sh(z+t))'}{1+2t^2Sh(z+t)} - \frac{2t(Ph(z+t))'}{1+tPh(z+t)}.
\end{equation*}
By means of \eqref{denom1} and \eqref{denom2}, $|1+2t^2Sh(z+t)|$ and $|1+tPh(z+t)|$ are bounded from below by 
a positive constant. Then,
this yields that 
\begin{equation*}
\begin{split}
 (2{\rm Re}z) |Ph_t(z)|^2  
\leq&  C ( (2{\rm Re}(z+t)) |Ph(z+t)|^2 \\
& \quad + (2{\rm Re}(z+t))^5 |(Sh(z+t))'|^2 + (2{\rm Re}(z+t))^3 |(Ph(z+t))'|^2\,)
\end{split}
\end{equation*}
for some positive constant $C$.
Making use of Proposition \ref{alpha2}, we also conclude by \eqref{PhCM} and \eqref{ShCM} that
\begin{equation}\label{htCM}
(2{\rm Re}z) |Ph_t(z)|^2 dxdy \in {\rm CM}(\mathbb H).
\end{equation}
This is equivalent to saying that $\log (h_t)'$ belongs to $\rm BMOA(\mathbb H)$.

Since $h$ admits quasiconformal extensions to $\mathbb C$ with $\lim_{z \to \infty}h(z)=\infty$
and so does $\varphi_{0,t}$ by Theorem \ref{chord} (iii),
we see that
$h_t = h\circ\varphi_{0,t}^{-1}$ can be quasiconformally extended to $\overline{\mathbb C}$ with $h_t(\infty)=\infty$. 
In fact, in virtue of property \eqref{htCM}, this further implies that $h_t$ admits such 
a particular quasiconformal extension $\widehat{h_t}$
that its
complex dilatation $\mu_t$ on $\mathbb H^*$ induces a Carleson measure 
\begin{equation}\label{mutCM}
|\mu_t(z)|^2/(-2{\rm Re}z)dxdy \in {\rm CM}(\mathbb H^*)
\end{equation}
(see \cite[Theorem 7.2]{Sh}). For example, it is known that
the variant of the Beurling--Ahlfors extension by heat kernel (see \cite[Theorem 4.2]{FKP}, \cite[Theorem 3.4]{WM-2})
realizes such an extension.

Mediated by the relation $\hat h(-t+iy) = h_t(iy)$, a map $\tilde h$ on $\mathbb H^*$ is composed by 
$$
\tilde h(z) := 
\begin{cases}
h(z), & \text {if}\;\; {\rm Re} z \geq 0,\\
  h_{-{\rm Re}z}(i{\rm Im}z) = h(z^*) + \frac{(2{\rm Re} z) h'(z^*)}{1 - ({\rm Re} z) Ph(z^*)},      
  & \text{if}\;\;  -t \leq {\rm Re} z < 0,\\
 \widehat{h_t}(z + t), & \text {if}\;\; {\rm Re} z < -t.
\end{cases}
$$
This is well-defined, and yields a quasiconformal extension of $h$ to $\mathbb C$ whose complex dilatation 
$\tilde \mu$ on $\mathbb H^*$ is given by
$$
\tilde \mu(z) = 
\begin{cases}
 - \frac{1}{2} (2{\rm Re}z)^2 Sh(z^*),      & \text{if}\;\;  -t \leq {\rm Re} z < 0,\\
 \mu_t(z + t), & \text {if}\;\; {\rm Re} z < -t.
\end{cases}
$$

We will show that $\tilde \mu$ induces a vanishing Carleson measure on $\mathbb H^*$ as in \eqref{carleson}.
For an interval $I \subset i\mathbb R$ with $|I| \leq t$, we see from \eqref{ShCM} that
$$
\frac{1}{|I|}\iint_{(-|I|,\, 0)\times I} \frac{|\tilde\mu(z)|^2}{(-2{\rm Re}z)} dxdy = \frac{1}{4|I|}\iint_{(0,\, |I|)\times I} (2{\rm Re}z)^3 |Sh(z)|^2 dxdy
$$
is bounded uniformly with respect to $I$, and tends to $0$ as $|I| \to 0$. 
For an interval $I \subset i\mathbb R$ with $|I| > t$, we have
\begin{equation}\label{bigbox}
\begin{split}
\frac{1}{|I|}\iint_{(-|I|,\, 0)\times I} \frac{|\tilde\mu(z)|^2}{-2{\rm Re}z} dxdy  
= & \frac{1}{4|I|}\iint_{(-t,0)\times I} (2{\rm Re}z)^2 |Sh(z)|^2 
dxdy\\
&+ \frac{1}{|I|}\iint_{(-|I|,-t]\times I} \frac{|\mu_t (z + t)|^2}{(-2{\rm Re}z)} dxdy.
\end{split}
\end{equation}
By \eqref{ShCM} and \eqref{mutCM}, it is not difficult to see that this is bounded uniformly with respect to $I$. 
Indeed, for the estimate of the first term of the right side of \eqref{bigbox}, we divide $[-t,0)\times I$ into
$n$ congruent rectangles $(-t,0)\times I_i$ $(i=1,\ldots,n)$ so that $nt \leq |I| < (n+1)t$. For the estimate of
the second term, we replace $(-|I|,-t]\times I$ with a larger square $(-|I|-t,-t]\times I$ and
translate it by $t$ along $x$-axis.
Consequently, $ |\tilde \mu (z)|^2/(-2{\rm Re}z) dxdy$ is a vanishing Carleson measure on $\mathbb H^*$. 
\end{proof}

\begin{remark}
As mentioned in the introduction, this result completes the characterization of the elements of
the VMO Teichm\"uller space on the upper half-plane developed by Shen \cite{Sh}.
In addition, we can show that the image of the space of Beltrami coefficients inducing 
vanishing Carleson measures on $\mathbb H^*$ by the Schwarzian derivative map coincides with
the intersection of the Bers embedding of the universal Teichm\"uller space $T$ and the space of the Schwarzian derivatives
inducing vanishing Carleson measures on $\mathbb H$.
This answers a question in \cite[Remark 5.2]{Sh}.
\end{remark}


\begin{thebibliography}{99}

\bibitem{ALR} D.S. Alexander, F. Lavernaro and A. Rosa,
Early Days in Complex Dynamics: A History of Complex Dynamics in One Variable During 1906-1942, History of Mathematics, AMS, 2011.
\bibitem{AIM} K. Astala, T. Iwaniec and G. Martin, Elliptic Partial Differential Equations and Quasiconformal Mappings in the Plane, Princeton Mathematical Series, Vol. 48, Princeton University Press, Princeton, NJ (2009). 
\bibitem{AZ} K. Astala and M. Zinsmeister, Teich\-m\"ul\-ler spaces and BMOA, Math. Ann. 289 (1991), 613--625.
\bibitem{AW62} L.V. Ahlfors and G. Weill, A uniqueness theorem for Beltrami equations, Proc. Amer. Math. Soc. 13 (1962), 975--978. 
\bibitem{BA} A. Beurling and L.V. Ahlfors, The boundary correspondence under quasiconformal mappings, Acta Math. 96 (1956), 125--142.
\bibitem{BCDM09} F. Bracci, M.D. Contreras and S. D\'iaz-Madrigal, Evolution families and the Loewner equation. II. Complex hyperbolic manifolds. Math. Ann. 344 (2009), 947--962. 
\bibitem{BCDM12} F. Bracci, M.D. Contreras and S. D\'iaz-Madrigal, Evolution families and the Loewner equation. I. The unit disc. J. Reine Angew. Math. 672 (2012), 1--37.
\bibitem{Be72} J. Becker, Loewner differentialgleichung und quasikonform fortsetzbare schlichte funktionen, J. Reine Angew. Math. 255 (1972), 23--43. 
\bibitem{Be80-1} J. Becker, Some inequalities for univalent functions with quasiconformal extensions, General Inequalities 2, Birkh\"auser Verlag, Basel, 1980. 
\bibitem{Be80} J. Becker, Conformal mappings with quasiconformal extension, Aspects of Contemporary Complex Analysis, pp.37--77,
Academic Press, London, 1980. 
\bibitem{Be87} J. Becker, On asymptotically conformal extension of univalent functions, Complex Variables 9 (1987), 109--120. 
\bibitem{BP} J. Becker and Ch. Pommerenke, \"Uber die quasikonforme fortsetzung schlichter funktionen, Math. Z. 161 (1978), 69--80.
\bibitem{Bi} C.J. Bishop, Weil-Petersson curves, conformal energies, $\beta$-numbers, and minimal surfaces, preprint. 
\bibitem{BP78} E. Berkson and H. Porta, Semigroups of holomorphic functions and composition operators, Mich. Math. J. 25 (1978), 101--115.

\bibitem{CDMG10} M.D. Contreras, S. D\'iaz-Madrigal and P. Gumenyuk, Loewner chains in the unit disk, Rev. Mat. Iberoam. 26 (2010), 975--1012. 
\bibitem{CDMG13} M.D. Contreras, S. D\'iaz-Madrigal and P. Gumenyuk, Loewner theory in annulus I: evolution families and differential equation, Trans. Amer. Math. Soc. 365 (2013), 2505--2543. 
\bibitem{DB85} L. De Branges, A proof of the Bieberbach conjecture, Acta Math. 154 (1985), 137--152. 
\bibitem{ES} M. Elin and D. Shoikhet, Linearization Models for Complex Dynamical Systems, Topics in Univalent Functions,
Functional Equations and Semigroup Theory, Operator Theory: Advances and Applications 208, Birkh\"auser, 2010.
\bibitem{FKP} R.A. Fefferman, C.E. Kenig and J. Pipher, The theory of weights and the Dirichlet problems for elliptic equations, Ann. of Math. 134 (1991), 65--124.
\bibitem{Ga} J.B. Garnett, Bounded Analytic Functions, Graduate Texts in Math. 236, Springer, 2006.
\bibitem{GB} V.V. Goryainov and I. Ba, Semigroup of conformal mappings of the upper half-plane into itself with
hydrodynamic normalization at infinity, Ukra\"in Math. Zh. 44 (1992), 1320--1329.
\bibitem{GH} P. Gumenyuk and I. Hotta, Chordal Loewner chains with quasiconformal extensions, Math. Z. 285 (2017), 1063--1089. 
\bibitem{GO} F.W. Gehring and B.G. Osgood, Uniform domains and the quasi-hyperbolic metric, J. Anal. Math. 36 (1979), 50--74. 
\bibitem{HWS} Y. Hu, L. Wu and Y. Shen, On symmetric homeomorphisms on the real line, Proc. Amer. Math. Soc.
146 (2018), 4255--4263.
\bibitem{Kuf43} P.P. Kufarev, On one-parameter families of analytic functions, Rec. Math. [Mat. Sbornik] N.S. 13 (1943), 87--118.
\bibitem{Kur} J. Kurzweil, Ordinary Differential Equations, Translated from the Czech by Michal Basch, Studies in Applied Mechanics, 13th edn. Elsevier, Amsterdam, 1986. 
\bibitem{Le} O. Lehto, Univalent Functions and Teichm\"uller Spaces, Graduate Texts in Math. 109, Springer, 1987.
\bibitem{LV} O. Lehto and K.I. Virtanen, Quasiconformal Mappings in the Plane. Springer-Verlag, 1973. 
\bibitem{Loe23} K. Loewner, Untersuchungen \"uber schlichte konforme abbildungen des einheitskreises. I.  Math. Ann. 89 (1923), 103--121.
\bibitem{Ne49} Z. Nehari, The Schwarzian derivative and schlicht functions, Bull. Amer. Math. Soc. 55 (1949), 545--551.
\bibitem{Pom65} Ch. Pommerenke, \"Uber die subordination analytischer funktionen, J. Reine Angew. Math. 218 (1965), 159--173. 
\bibitem{Pom75} Ch. Pommerenke, Univalent Functions, Vandenhoeck \& Ruprecht, 1975.
\bibitem{Pom} Ch. Pommerenke, Boundary Behaviour of Conformal Maps, Springer, 1992.
\bibitem{Sch00} O. Schramm, Scaling limits of loop-erased random walks and uniform spanning trees, Israel J. Math. 118 (2000), 221--288. 
\bibitem{Sh} Y. Shen, VMO--Teichm\"uller space on the real line, Ann. Fenn. Math. 47 (2022), 57--82. 
\bibitem{SW} Y. Shen and H. Wei, Universal Teichm\"uller space and BMO, Adv. Math. 234 (2013) 129--148.
\bibitem{TV} P. Tukia and J. V\"{a}is\"{a}l\"{a}, Quasisymmetric embeddings of metric spaces, Ann. Acad. Sci. Fenn. Ser. A I Math. 5 (1980),
97--114.
\bibitem{WM-1} H. Wei and K. Matsuzaki, Teichm\"uller spaces of generalized symmetric homeomorphisms, Proc. Amer. Math. Soc. Ser. B 7 (2020), 52--66. 
\bibitem{WM-2} H. Wei and K. Matsuzaki, Beurling--Ahlfors extension by heat kernel, ${\rm A}_\infty$-weights for VMO, and vanishing Carleson measures, Bull. London Math. Soc. 53 (2021), 723--739.
\bibitem{WM-3} H. Wei and K. Matsuzaki, The $p$-Weil--Petersson Teichm\"uller space and the quasiconformal extension of curves, J. Geom. Anal. 32 (2022), 213. 

\end{thebibliography}
\end{document}